\documentclass[final,onefignum,onetabnum]{siamart250211}
\usepackage{amsmath,amssymb}
\usepackage{color}
\usepackage{graphics,graphicx}
\usepackage{bm}
\usepackage{bbm}
\usepackage{cite}
\usepackage{epsfig}
\usepackage{subfigure}
\usepackage{psfrag}
\usepackage{epstopdf}
\usepackage{verbatim}
\usepackage{fancyhdr}
\usepackage{CJKutf8}
\usepackage{mathrsfs}
\usepackage{mathtools}
\usepackage{listings}
\usepackage{xcolor}
\usepackage{stmaryrd}
\usepackage{tikz}
\usepackage{float}

\usepackage{caption}
\usepackage{hyperref}
\hypersetup{hypertex=true,
colorlinks=true,
linkcolor=blue,
anchorcolor=blue,
citecolor=blue}

\numberwithin{equation}{section}
\newtheorem{assumption}{Assumption}
\def\tbar{{|\hspace{-.02in}|\hspace{-.02in}|}}

\def\cT{\mathcal{T}}

\def\RM{{\operatorname{RM}}}
\newtheorem{exmp}{Example}[section]

\allowdisplaybreaks[4]

\title{A Generalized Weak Galerkin Method for Linear Elasticity with Nonpolynomial Approximations}
\author{
Junping Wang
\thanks{Directorate for Mathematical and Physical Sciences, 
           U.S. National Science Foundation, 
           Alexandria, VA 22314, USA
           (jwang@nsf.gov)}
\and
Yue Wang
\thanks{Petrochina (Beijing) Digital Intelligence Research Institute Co., Ltd., Beijing, China (wendy\_680@qq.com).}}

\begin{document}

\maketitle
\begin{abstract}
This paper presents a generalized weak Galerkin (gWG) finite element method for linear elasticity problems on general polygonal and polyhedral meshes. The proposed framework is flexible and efficient, allowing for the use of nonpolynomial approximating functions. The generalized weak differential operators are defined as an element-level correction of the classical differential operators accounting for boundary discontinuities. This construction reduces computational cost and provides greater flexibility than standard weak Galerkin formulations. The gWG framework naturally accommodates arbitrary finite-dimensional approximation spaces, including nonpolynomial activation-based spaces with randomly selected parameters. Error equations and error estimates are established for the proposed method. Numerical experiments demonstrate that the method is locking-free, robust with respect to mesh geometry, and effective on general polygonal and polyhedral partitions. In particular, activation-based interior approximation spaces exhibit convergence behavior comparable to that of classical polynomial spaces.
\end{abstract}




\section{Introduction}
In this paper, we consider the linear elasticity problem
\begin{subequations}\label{eq:Problem}
\begin{align}
\label{le_1}-\nabla \cdot \sigma(\bm{u}) &=\bm{f}, & & \text { in } \Omega,\\
\label{le_2}\bm{u} &=\bm{g}, & & \text { on } \partial \Omega,
\end{align}
\end{subequations}
where $\Omega\subset\mathbb{R}^d\,(d=2,3)$ is the polygonal/polyhedral domain, $\bm{u}$ denotes the displacement field, $ \sigma(\bm{u})$ is the stress tensor, $\bm{f}$ is the body force, and $\bm{g}$ is the prescribed boundary displacement. The divergence $\nabla \cdot \sigma(\bm{u})$ is taken tow-wise. For homogeneous isotropic linear elasticity,
$$
\sigma(\bm{u})=2\mu\varepsilon(\bm{u})+\lambda\nabla\cdot\bm{u}\,{\rm I},
$$
where $\varepsilon(\bm{u})=\frac{1}{2}(\nabla \bm{u}+\nabla \bm{u}^\mathrm{T})$ is the strain tensor. $\lambda$ and $\mu$ are the Lam'{e} constants satisfying $0 < \alpha_1 \leq \lambda \leq \alpha_2 < \infty$ and $0 < \beta_1 \leq \mu \leq \beta_2 < \infty$, and ${\rm I}$ is the identity matrix.

\smallskip
Linear elasticity models arise in a wide variety of engineering and scientific applications. Numerous numerical methods have been developed to address issues of stability, robustness, and locking, especially for nearly incompressible media. For instance, 
in \cite{fuentes_coupled_2017}, broken formulations were shown to naturally dispose of many complications arising with the compatibility and well-posedness of coupled variational formulations. Nonconforming virtual element method for linear elasticity problems was introduced in \cite{zhang_nonconforming_2019} to derive a locking-free formulation on polygonal or polyhedral meshes, although the stability of the bilinear form was weaker than the usual one presented in the literature when $\lambda$ becomes large. A mixed discontinuous Galerkin method with strongly imposed symmetry was introduced in \cite{wang_mixed_2020} by Wang, Wu and Xu, where a priori error estimates were established using a mesh-dependent norm. A data-driven Physics-Informed Neural Network framework \cite{roy2023deep} was presented by Arunabha et al. for linear elasticity problems, featuring a multi-objective loss function that incorporates governing PDEs, boundary conditions, and constitutive relations. The proposed methodology employs multiple densely connected
independent artificial neural networks to obtain accurate solutions. Later, a mesh-based PINN approach \cite{wang2024m} was introduced by Wang et al. that integrated the concept of domain discretization from the finite element method to guide and constrain the neural network's optimization process, thereby enhancing convergence and stability. The advantages of the method include reduced reliance on complex gradient computations and automatic differentiation, and handled problems with unknown boundary conditions validated through solid mechanics simulations.

\smallskip
The weak Galerkin finite element method (WG-FEM), introduced by Wang and Ye \cite{wang_weak_2013}, replaces classical differential operators by discrete weak operators. Several versions of WG-FEM have been developed for linear elasticity. Wang et al.~\cite{wang_locking-free_2016} introduced a locking-free formulation under the requirement that boundary unknowns include the rigid motion space. Lowest-order WG schemes on rectangular and brick meshes were studied in \cite{harper_lowest-order_2019,yi_lowest-order_2019}, based on the combination ${[P_0(T)]^d,[P_0(\partial T)]^d,RT^d_0(T),P_0(T)}$. Later, a simplified WG formulation with $P_1/P_0$ approximations and tangential stabilization was proposed in \cite{liu_locking-free_2022}, although evaluation of the stabilizer is relatively complex. Recently, a penalty-free any-order WG method on convex quadrilateral meshes was developed in \cite{wang_penalty-free_2023}, and a lowest-order WG-FEM on convex polygonal grids was presented in \cite{wang_lowest-order_2024}. A stabilizer-free WG-FEM \cite{wang2024simplifiedweakgalerkinmethods} was introduced by Wang and Zhang that utilized bubble functions on general polygonal/polyhedral meshes (without convexity constraints), which was symmetric, positive definite, and achieved optimal convergence rates.

\smallskip
In this work, we propose the generalized weak Galerkin (gWG) method for linear elasticity on general polygonal and polyhedral meshes with polynomial or non-polynomial approximations. The gWG framework has been successfully applied to elliptic equations \cite{li_generalized_2024}, Oseen equations \cite{qi_generalized_2024}, Stokes equations \cite{qi_generalized_2023}, and biharmonic problems \cite{li_generalized_2023}. The key feature of gWG is that the generalized weak gradient/divergence is defined as the sum of the classical gradient/divergence and the solution of a simple local correction problem corresponding to the discontinuity along element boundaries. As a result, the computational cost of evaluating weak differential operators is significantly reduced compared to standard WG-FEM, and moreover, the method admits great flexibility in selecting stabilizer parameters $(\gamma,R_b)$ and local approximation spaces. We derive the corresponding error equations and error estimates for arbitrary finite-dimensional spaces. In particular, our numerical results demonstrate that activation-based approximation spaces with random parameters can be used effectively inside elements, achieving convergence rates comparable to classical polynomial spaces in many cases, and preserving uniform locking-free performance. Compared with \cite{wang_locking-free_2016}, the use of non-polynomial function spaces resolves the difficulty of well-posedness when taking $V^b(T)=[P_0(T)]^d$ on element boundaries.

\smallskip
The remainder of this paper is organized as follows. Section \ref{gWG section} introduces the generalized weak operators and the gWG algorithm, and establishes well-posedness. Section \ref{convergence} provides the error analysis. Section \ref{EX} presents numerical experiments illustrating the convergence behavior and locking-free performance for various choices of approximation spaces, stabilizer parameters, and mesh types.

\smallskip
For any polygonal/polyhedral subdomain $D\subseteq\Omega$, let $H^{s}(D)$ be the standard Sobolev space with $s\ge0$, with inner product $(\cdot,\cdot)_{s,D}$, norm $|\cdot|_{s,D}$, and semi-norm $|\cdot|_{s,D}$. When $s=0$, the subscript $s$ is omitted.

\section{Numerical Scheme}\label{gWG section}
\smallskip
Let $\mathcal{T}_{h}$ be a polygonal/polyhedral partition of the domain $\Omega$, and denote by $\mathcal{E}_h$ the set of all edges/faces in $\mathcal{T}_{h}$. For each element $T \in \mathcal{T}_{h}$, denote by $h_T$ its diameter, and define the mesh size $h=\max_{T \in \mathcal{T}_{h}} h_T$. On each $T \in \mathcal{T}_{h}$, we introduce the local weak function space 
\begin{equation}\label{EQ:local-space}
V(T)=\{\bm{v}=\{\bm{v}_0,\bm{v}_b\}: \bm{v}_0\in V^0(T),\ \bm{v}_b\in V^b(e),~e\subset \partial T\},
\end{equation}
where $V^0(T)$ and $V^b(e)$ are finite dimensional spaces possessing standard approximation properties.

The global weak finite element space $V_h$ is obtained by assembling $V(T)$ over all $T\in \mathcal{T}_h$, subject to the requirement that each $\bm{v}_b$ attains a unique value on every interior edge/face $e \in \mathcal{E}_h$. More precisely,
\begin{equation*}
V_h=\{\bm{v}=\{\bm{v}_0,\bm{v}_b\}: \bm{v}|_T\in V(T),\ \bm{v}_b|_{\partial T\cap e}=\bm{v}_b|_{\partial T^*\cap e}, \ T \in{\cT}_h\},
\end{equation*}
where $T$ and $T^*$ denote two neighboring elements sharing the common interface $e$ (see Figure \ref{Fig_WG}). The subspace with zero boundary value is defined by 
$$
V_h^0=\{\bm{v}: \ \bm{v}\in V_h,~\bm{v}_b=\bm{0}~\text{on}~\partial \Omega\}.
$$
\begin{figure}[!ht]{}
  \centering  \centerline{\includegraphics[scale=0.09]{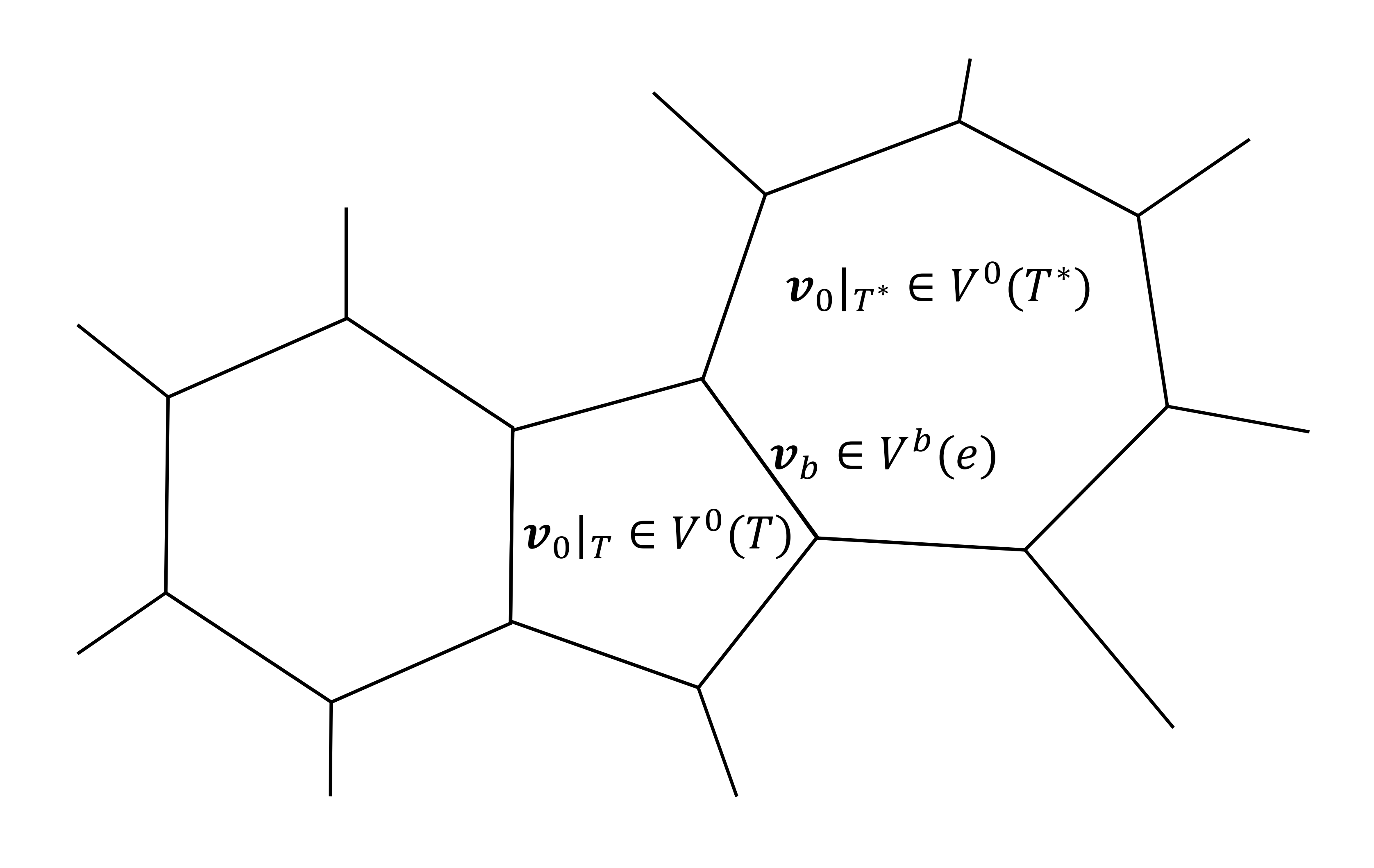}}
  \caption{Two neighboring elements $T$ and $T^*$ sharing an interface $e=\partial T\cap\partial T^\in\mathcal{E}_h$, where $\bm{v}_b$ is single-valued.}
\label{Fig_WG}
\end{figure}

For each $T\in\mathcal{T}_h$, let $\RM(T)=\{\bm{a}+\eta\bm{x}: \bm{a}\in\mathbb{R}^d, \eta\in so(d)\}$ be the space of rigid motions on $T$, where $so(d)$ is the space of $d\times d$ skew-symmetric matrices. The restriction of $\RM(T)$ to a face $e\subset\partial T$ represents the space of rigid motion on $e$, and is denoted by $\RM(e)$.

We introduce a linear operator $R_b: [L^2(\mathcal{E}_h)]^{d} \to [L^2(\mathcal{E}_h)]^{d}$ that satisfies the following two assumptions.

\begin{assumption}[Invariance under Rigid Motion]\label{as:rigidmotioninvariance}
 $R_b$ preserves any displacement belonging to the restriction of $\RM(T)\cap V^0(T)$ on $e\subset\partial T$. That is, $R_b(\phi|_{e})=\phi|_e$ for all $\phi\in \RM(T)\cap V^0(T)$ and $e\in\mathcal{E}_h$.
 \end{assumption}

\begin{assumption}[Injectivity]\label{as:one2one}
$R_b$ is one-to-one on each space $V^b(e)$, $e\in\mathcal{E}_h$.
 \end{assumption}

\smallskip
\begin{definition}\label{def gweak grad}
(Generalized Weak Gradient)
Let $G_1(T)\subset [L^2(T)]^{d\times d}$ be given. For $\bm{v}\in V(T)$, the generalized weak gradient on $T$, denoted as ${\nabla}_{g,T}\bm{v}$, is the unique function in $\nabla V^0(T)+G_1(T)$ such that 
\begin{equation}\label{weak gradient-1}
\begin{split}
{\nabla}_{g,T}\bm{v}={\nabla}\bm{v}_0+{\delta}_{g_1,T} \bm{v},
\end{split}
\end{equation}
where ${\nabla}\bm{v}_0$ is the classical gradient of $\bm{v}_0$ and ${\delta}_{g_1,T} \bm{v}\in G_1(T)$ satisfies
\begin{equation}\label{weak gradient-2}
({\delta}_{g_1,T} \bm{v},\boldsymbol{\psi})_T=\langle R_b(\bm{v}_b-\bm{v}_0),\boldsymbol{\psi} \mathbf{n}\rangle_{\partial T},\quad \forall\boldsymbol{\psi}\in G_1(T),
\end{equation}
with $\mathbf{n}$ denoting the outward normal of $\partial T$.
\end{definition}

\smallskip
\begin{definition} (Generalized Weak Divergence)
Let $G_2(T)\subset L^2(T)$ be given. For $\bm{v}\in V(T)$, the generalized weak divergence is defined as the unique function ${\nabla}_{g,T}\cdot\bm{v} \in \nabla\cdot V^0(T) + G_2(T)$ satisfying
\begin{equation}\label{weak divergence-1}
\begin{split}
{\nabla}_{g,T}\cdot\bm{v}={\nabla}\cdot\bm{v}_0+{\delta}_{g_2,T} \bm{v},
\end{split}
\end{equation}
where ${\delta}_{g_2,T} \bm{v}\in G_2(T)$ solves
\begin{equation}\label{weak divergence-2}
({\delta}_{g_2,T} \bm{v},\phi)_T=\langle R_b(\bm{v}_b-\bm{v}_0),\phi \mathbf{n}\rangle_{\partial T},\quad \forall\phi\in G_2(T).
\end{equation}
\end{definition}

\smallskip
For the specific choice $\bm{v}_0\in[P_k(T)]^d$, $\bm{v}_b\in[P_{k-1}(e)]^d$, $G_1(T)=[P_{k-1}(T)]^{d\times d}$, $G_2(T)= P_{k-1}(T)$, and $R_b=Q_b$ (the $L^2$ projection operator) , the generalized weak gradient and divergence recover the classical weak formulations of \cite{wang_weak_2013,wang_locking-free_2016}.

\smallskip
For convenience, we adopt the notations:
$$
\begin{aligned}
(\cdot,\cdot)_{\cT_h}=&\sum_{T\in \cT_h}(\cdot,\cdot)_T,\,\,&\langle \cdot,\cdot \rangle_{\partial \cT_h}=&\sum_{T\in \cT_h}\langle \cdot,\cdot \rangle_{\partial T},\\
\|\cdot\|_{\cT_h}=&\sum_{T\in \cT_h}\|\cdot\|_T,\,\,&\|\cdot\|_{\partial \cT_h}=&\sum_{T\in \cT_h}\|\cdot\|_{\partial T},\\
({\nabla}_{g}\bm{v})|_T =& \ {\nabla}_{g,T}\bm{v},\qquad &({\nabla}_{g}\cdot\bm{v})|_T =&\ {\nabla}_{g,T}\cdot\bm{v}.
\end{aligned}
$$
\smallskip
\begin{algorithm}
\caption{gWG Algorithm}
\label{linear elasticity algorithm}
Find $\bm{u}_h\in V_h$ such that $\bm{u}_h|_{\partial \Omega}=Q_b \bm{g}$ and 
\begin{equation}\label{gWG_algorithm}
a(\bm{u}_h,\bm{v})+s(\bm{u}_h,\bm{v})=(\bm{f},\bm{v}_0)\quad \forall \ \bm{v}\in V_h^0,
\end{equation}
where 
\begin{equation*}
\begin{split}
a(\bm{u}_h,\bm{v})=&(2\mu \varepsilon_g(\bm{u}_h), \varepsilon_g (\bm{v}))_{\cT_h}+(\lambda \nabla_g\cdot \bm{u}_h, \nabla_g \cdot \bm{v})_{\cT_h},\\
s(\bm{u}_h,\bm{v}) =& \sum_{T\in \cT_h}\rho h_T^{\gamma}\langle R_b(\bm{u}_0-\bm{u}_b),R_b(\bm{v}_0-\bm{v}_b) \rangle_{\partial T},\\
\varepsilon_g(\bm{v})=&(\nabla_g \bm{v}+\nabla_g \bm{v}^\mathrm{T})/2.
\end{split}
\end{equation*}
Here $\rho>0$ and $\gamma$ are fixed parameters,  typically chosen as $\rho=1$ and $\gamma=-1$. The operator $Q_b$ denotes the standard $L^2$ projection onto $V^b(e)$ for each $e\in\mathcal{E}_h$.
\end{algorithm}

The generalized weak Galerkin formulation above is a discrete weak formulation driven entirely by weak derivatives, admitting general (possibly non-polynomial) local approximation spaces. In this framework, the classical strong gradient and divergence operators are replaced by their generalized weak counterparts defined elementwise on user-selected approximation spaces. The bilinear form $a(\cdot,\cdot)$ represents the discrete linear elasticity operator, where weak gradients, when symmetrized, yields a consistent approximation of the strain tensor. The stabilization term $s(\cdot,\cdot)$ imposes control on the jump of the weak finite element unknowns across inter-element boundaries. In particular, it enforces a weak continuity mechanism for the interior–boundary variable pairs $(\bm{v}_0,\bm{v}_b)$ so that the discrete space behaves as a stable approximation to the continuous $H^1$ displacement field. Overall, the gWG scheme maintains the flexibility of discontinuous finite element approximations, including those built from non-polynomial functions, while still achieving a robust and physically consistent discretization of linear elasticity.


\begin{lemma}\label{second korn inequ}
\cite{wang_locking-free_2016} (The second Korn's inequality) Assume $\Omega$ is connected, bounded, and has a Lipschitz boundary. Let $\Gamma_{1} \subset \partial \Omega$ be a nontrivial portion of the boundary $\partial \Omega$ with dimension $d-1$. For any $1 \leq p\leqslant p^*$, there exists a constant $C$ such that
$$
\|\bm{v}\|_{1} \leq C\left(\|\varepsilon(\bm{v})\|_{0}+\|\bm{v}\|_{L^{p}\left(\Gamma_{1}\right)}\right),\,\, \forall \bm{v} \in\left[H^{1}(\Omega)\right]^{d},
$$
where $p^*=\infty$ if $d=2$ and $p^*=4$ if $d=3$.
\end{lemma}

\begin{theorem}\label{well posedness}
The generalized weak Galerkin scheme \eqref{gWG_algorithm} admits a unique solution.
\end{theorem}
\begin{proof}
Since the number of unknowns equals the number of equations, it suffices to show uniqueness. Let $\bm{u}_h^{(i)}=\{\bm{u}_0^{(i)},\bm{u}_b^{(i)}\}\in V_h$ $(i=1,2)$ be two solutions of \eqref{gWG_algorithm}. Then $\bm{u}_b^{(i)}=Q_b \bm{g}$ and
$$
a(\bm{u}_h^{(i)},\bm{v})+s(\bm{u}_h^{(i)},\bm{v})=(\bm{f},\bm{v}_0),\,\,\,\forall \bm{v}\in V_h^0.
$$ 
Set $\bm{w}=\bm{u}_h^{(1)}-\bm{u}_h^{(2)}$. Then $\bm{w}\in V_h^{0}$ and satisfies
$$
a(\bm{w},\bm{v})+s(\bm{w},\bm{v})=0,\quad \forall\  \bm{v}\in V_h^0.
$$
Taking $\bm{v}=\bm{w}$ gives
$$
(2\mu \varepsilon_g(\bm{w}), \varepsilon_g (\bm{w}))_{\cT_h}+(\lambda \nabla_g\cdot \bm{w}, \nabla_g \cdot \bm{w})_{\cT_h}+\sum_{T\in \cT_h}\rho h_T^{\gamma}\langle R_b(\bm{w}_0-\bm{w}_b),R_b(\bm{w}_0-\bm{w}_b) \rangle_{\partial T}=0,
$$
which implies 
\begin{equation}\label{solution_uni_1}
\varepsilon_g(\bm{w})=\bm{0},
\end{equation}
and
\begin{equation}\label{solution_uni_2}
R_b(\bm{w}_0-\bm{w}_b)=\bm{0}.
\end{equation}

From Definition \ref{def gweak grad}, \eqref{solution_uni_2} and \eqref{solution_uni_1} we obtain ${\nabla}_{g}\bm{w}={\nabla}\bm{w}_0$, and hence
$$
\varepsilon(\bm{w}_0)=\varepsilon_g(\bm{w})=\bm{0}.
$$ 
Thus $\bm{w}_0\in \RM(T)\cap V^0(T)\subset [P_1(T)]^d$ for each $T$. By Assumption \ref{as:rigidmotioninvariance} and \eqref{solution_uni_2},
$$
\bm{w}_0|_e = R_b(\bm{w}_0|_e)= R_b(\bm{w}_b),
$$
which yields the continuity of $\bm{w}_0$ in $\Omega$. 

Since $\bm{w}\in V_h^0$, Lemma \ref{second korn inequ} implies $\bm{w}_0=\bm{0}$. Then \eqref{solution_uni_2} yields $R_b( \bm{w}_b ) =0$. By the injective Assumption \ref{as:one2one}, we have $\bm{w}_b=0$. Therefore, $\bm{u}_h^{(1)}=\bm{u}_h^{(2)}$, which completes the proof of the theorem.

\end{proof}

\section{Error Estimates}\label{convergence}
In this section, we derive the error equation and establish error estimates for the numerical solutions arising from the gWG scheme \eqref{gWG_algorithm}. 

For each element $T$, let $Q_h = \{Q_0,Q_b\}$ denote the $L^2$ projection operator from $[L^2(T)]^d$ onto $V(T)$. Similarly, let $\mathbb{Q}_h$ and $\mathcal{Q}_h$ denote the $L^2$ projections onto $G_1(T)$ and $G_2(T)$, respectively. We denote the identity operator by $I$, and define the error 
$$
\bm{e}_h:=Q_h \bm{u}-\bm{u}_h.
$$

\begin{lemma}
Let $\Phi\in [H^1(T)]^d$. Then for all  $\boldsymbol{\psi}\in G_1(T)$, the following identity holds:
\begin{equation}\label{EQ:Qh-property-epsilon}
\begin{split}
(\varepsilon_g(Q_h \Phi),\boldsymbol{\psi})_T = & (\varepsilon(\Phi),\boldsymbol{\psi})_T + (\varepsilon(Q_0\Phi-\Phi), \boldsymbol{\psi})_T\\
&+\frac12 \left\langle R_b(Q_b\Phi-\Phi),(\boldsymbol{\psi}+\boldsymbol{\psi}^{\mathrm{T}}) \mathbf{n}\right\rangle_{\partial T}\\
& + \frac12\left\langle R_b(\Phi-Q_0\Phi),(\boldsymbol{\psi}+\boldsymbol{\psi}^{\mathrm{T}})\mathbf{n}\right\rangle_{\partial T}.
\end{split}
\end{equation}
Moreover, for all $\phi\in G_2(T)$, it holds that
\begin{equation}\label{EQ:Qh-property-wd}
\begin{split}
({\nabla}_g \cdot Q_h\Phi,\phi)_T = & ({\nabla}\cdot\Phi,\phi)_T + (\nabla\cdot(Q_0\Phi-\Phi), \phi)_T\\
&+\langle R_b(Q_b\Phi-\Phi),\phi \mathbf{n}\rangle_{\partial T} + \langle R_b(\Phi-Q_0\Phi),\phi \mathbf{n}\rangle_{\partial T}.
\end{split}
\end{equation}
\end{lemma}

\begin{proof}
From \eqref{weak gradient-1} and \eqref{weak gradient-2}, we obtain
\begin{equation}\label{EQ:prof-Qh-property-1}
\begin{split}
({\nabla}_g Q_h\Phi,\boldsymbol{\psi})_T =& ({\nabla} Q_0\Phi+{\delta}_{g_1} Q_h\Phi,\boldsymbol{\psi})_T\\
=&({\nabla} Q_0\Phi,\boldsymbol{\psi})_T+\langle R_b(Q_b\Phi-Q_0\Phi),\boldsymbol{\psi} \mathbf{n}\rangle_{\partial T}\\
=&({\nabla} \Phi,\boldsymbol{\psi})_T + ({\nabla} (Q_0\Phi-\Phi),\boldsymbol{\psi})_T\\
&+\langle R_b(Q_b\Phi-\Phi),\boldsymbol{\psi} \mathbf{n}\rangle_{\partial T} + \langle R_b(\Phi-Q_0\Phi),\boldsymbol{\psi} \mathbf{n}\rangle_{\partial T}.
\end{split}
\end{equation}
Analogously, the following equation holds:
\begin{equation}\label{EQ:prof-Qh-property-2}
\begin{split}
({\nabla}_g Q_h\Phi^{\mathrm{T}},\boldsymbol{\psi})_T = & ({\nabla}\Phi^{\mathrm{T}},\boldsymbol{\psi})_T + (\nabla(Q_0\Phi-\Phi)^{\mathrm{T}}, \boldsymbol{\psi})_T\\
&+\langle R_b(Q_b\Phi-\Phi),\boldsymbol{\psi}^{\mathrm{T}}\mathbf{n}\rangle_{\partial T} + \langle R_b(\Phi-Q_0\Phi),\boldsymbol{\psi}^{\mathrm{T}}\mathbf{n}\rangle_{\partial T}.
\end{split}
\end{equation}
Combining \eqref{EQ:prof-Qh-property-1} with \eqref{EQ:prof-Qh-property-2} yields \eqref{EQ:Qh-property-epsilon}. 

Next, from \eqref{weak divergence-1} and \eqref{weak divergence-2}, we have
\begin{equation}\label{EQ:prof-Qh-property-3}
\begin{split}
({\nabla}_g \cdot Q_h\Phi,\phi)_T =& ({\nabla} \cdot Q_0\Phi+{\delta}_{g_2} Q_h\Phi,\phi)_T\\
=&({\nabla}\cdot Q_0\Phi,\phi)_T+\langle R_b(Q_b\Phi-Q_0\Phi),\phi\mathbf{n}\rangle_{\partial T}\\
=&({\nabla}\cdot \Phi,\phi)_T + ({\nabla}\cdot (Q_0\Phi-\Phi,\phi)_T\\
&+\langle R_b(Q_b\Phi-\Phi),\phi\mathbf{n}\rangle_{\partial T} + \langle R_b(\Phi-Q_0\Phi),\phi\mathbf{n}\rangle_{\partial T},
\end{split}
\end{equation}
which verifies \eqref{EQ:Qh-property-wd}.
\end{proof}

\begin{lemma}
Let $\bm{u}\in [H^1(\Omega)]^d$ and $\bm{u}_h$ be the solutions of \eqref{eq:Problem} and \eqref{gWG_algorithm}, respectively. Then the error $\bm{e}_h=Q_h \bm{u}-\bm{u}_h$ belongs to $V_h^0$ and satisfies
\begin{equation}\label{error equ}
a(\bm{e}_h,\bm{v}) + s(\bm{e}_h,\bm{v}) = \ell(\bm{u},\bm{v}),\quad \forall \bm{v} \in V_h^0,
\end{equation}
where
\begin{equation}
\begin{split}
\ell(\bm{u},\bm{v}) = & (\varepsilon(Q_0 \bm{u}),2\mu (I-\mathbb{Q}_h) \varepsilon_g(\bm{v}))_{\mathcal{T}_h} + ({\nabla} \cdot Q_0\bm{u},\lambda (I-\mathcal{Q}_h) \nabla_g \cdot \bm{v})_{\mathcal{T}_h} + s(Q_h \bm{u},\bm{v})\\
&+ \langle R_b(\bm{v}_0-\bm{v}_b),(2\mu (I-\mathbb{Q}_h) \varepsilon (\bm{u}) + \lambda (I-\mathcal{Q}_h) ({\nabla}\cdot\bm{u})\,{\rm I}) \mathbf{n}\rangle_{\partial \mathcal{T}_h}\\
&+ \langle \sigma(\bm{u})\mathbf{n},(I-R_b)(\bm{v}_0-\bm{v}_b) \rangle_{\partial \mathcal{T}_h} \\
&+ ( 2\mu (\mathbb{Q}_h-I) \varepsilon(\bm{u}),\varepsilon (\bm{v}_0))_{\mathcal{T}_h} + (\lambda (\mathcal{Q}_h-I) {\nabla}\cdot\bm{u},\nabla \cdot \bm{v}_0)_{\mathcal{T}_h} \\
&+ (\varepsilon(Q_0\bm{u}-\bm{u}), 2\mu \mathbb{Q}_h \varepsilon_g(\bm{v}))_{\mathcal{T}_h} + (\nabla\cdot(Q_0\bm{u}-\bm{u}), \lambda \mathcal{Q}_h \nabla_g \cdot \bm{v})_{\mathcal{T}_h}\\
&+\langle R_b(Q_b\bm{u}-\bm{u}),(2\mu \mathbb{Q}_h \varepsilon_g(\bm{v}) + \lambda \mathcal{Q}_h \nabla_g \cdot \bm{v} \,{\rm I})\mathbf{n}\rangle_{\partial {\mathcal{T}_h}} \\
&+ \langle R_b(\bm{u}-Q_0\bm{u}),(2\mu \mathbb{Q}_h \varepsilon_g(\bm{v}) + \lambda \mathcal{Q}_h \nabla_g \cdot \bm{v} \,{\rm I})\mathbf{n}\rangle_{\partial {\mathcal{T}_h}}.\\
\end{split}
\end{equation}
\end{lemma}
\begin{proof}
Taking $\Phi=\bm{u}$ and $\boldsymbol{\psi}=2\mu \mathbb{Q}_h \varepsilon_g(\bm{v})$ in \eqref{EQ:Qh-property-epsilon} and summing over all the elements yields 
\begin{equation}\label{error equ 1}
\begin{split}
&(\varepsilon_g(Q_h \bm{u}),2\mu \mathbb{Q}_h \varepsilon_g(\bm{v}))_{\mathcal{T}_h} \\
= & (\varepsilon(\bm{u}),2\mu \mathbb{Q}_h \varepsilon_g(\bm{v}))_{\mathcal{T}_h} + (\varepsilon(Q_0\bm{u}-\bm{u}), 2\mu \mathbb{Q}_h \varepsilon_g(\bm{v}))_{\mathcal{T}_h}\\
&+\langle R_b(Q_b\bm{u}-\bm{u}),2\mu \mathbb{Q}_h \varepsilon_g(\bm{v})\cdot\mathbf{n}\rangle_{\partial {\mathcal{T}_h}} + \langle R_b(\bm{u}-Q_0\bm{u}),2\mu \mathbb{Q}_h \varepsilon_g(\bm{v})\mathbf{n}\rangle_{\partial {\mathcal{T}_h}}.
\end{split}
\end{equation}
Using the $L^2$ projection properties, \eqref{weak gradient-1}-\eqref{weak gradient-2}, integration by parts, together with $\langle 2\mu \varepsilon(\bm{u})\mathbf{n},\bm{v}_b \rangle_{\partial \mathcal{T}_h} = 0$, yields
\begin{equation}\label{error equ 2}
\begin{split}
&(\varepsilon(\bm{u}),2\mu \mathbb{Q}_h \varepsilon_g(\bm{v}))_{\mathcal{T}_h} \\
= & ( 2\mu \mathbb{Q}_h \varepsilon(\bm{u}),\varepsilon_g(\bm{v}))_{\mathcal{T}_h}\\
= & ( 2\mu \mathbb{Q}_h \varepsilon(\bm{u}),\varepsilon (\bm{v}_0))_{\mathcal{T}_h} + \left( 2\mu \mathbb{Q}_h \varepsilon(\bm{u}),\frac{\delta_{g_1}\bm{v}+{\delta_{g_1} \bm{v}}^{\mathrm{T}}}{2}\right)_{\mathcal{T}_h}\\
=& ( 2\mu (\mathbb{Q}_h-I) \varepsilon(\bm{u}),\varepsilon (\bm{v}_0))_{\mathcal{T}_h} + \langle R_b(\bm{v}_b-\bm{v}_0),2\mu \mathbb{Q}_h \varepsilon (\bm{u})\mathbf{n}\rangle_{\partial \mathcal{T}_h}+( 2\mu \varepsilon(\bm{u}),\varepsilon (\bm{v}_0))_{\mathcal{T}_h} \\
=& ( 2\mu (\mathbb{Q}_h-I) \varepsilon(\bm{u}),\varepsilon (\bm{v}_0))_{\mathcal{T}_h} + \langle R_b(\bm{v}_b-\bm{v}_0),2\mu \mathbb{Q}_h \varepsilon (\bm{u})\mathbf{n}\rangle_{\partial \mathcal{T}_h} \\
&- ( \nabla \cdot (2\mu \varepsilon(\bm{u})),\bm{v}_0)_{\mathcal{T}_h} + \langle 2\mu \varepsilon(\bm{u})\mathbf{n},(\bm{v}_0-\bm{v}_b) \rangle_{\partial \mathcal{T}_h}\\
=& - ( \nabla \cdot (2\mu \varepsilon(\bm{u})),\bm{v}_0)_{\mathcal{T}_h} + ( 2\mu (\mathbb{Q}_h-I) \varepsilon(\bm{u}),\varepsilon (\bm{v}_0))_{\mathcal{T}_h} \\
&+ \langle 2\mu \varepsilon(\bm{u})\mathbf{n},(I-R_b)(\bm{v}_0-\bm{v}_b) \rangle_{\partial \mathcal{T}_h} + \langle R_b(\bm{v}_0-\bm{v}_b),2\mu (I-\mathbb{Q}_h) \varepsilon (\bm{u})\mathbf{n}\rangle_{\partial \mathcal{T}_h}.
\end{split}
\end{equation}
Similarly, taking $\Phi=\bm{u}$ and $\phi = \lambda \mathcal{Q}_h \nabla_g \cdot \bm{v}$ in \eqref{EQ:Qh-property-wd} gives
\begin{equation}\label{error equ 3}
\begin{split}
&({\nabla}_g \cdot Q_h\bm{u},\lambda \mathcal{Q}_h \nabla_g \cdot \bm{v})_{\mathcal{T}_h} \\
=& ({\nabla}\cdot\bm{u},\lambda \mathcal{Q}_h \nabla_g \cdot \bm{v})_{\mathcal{T}_h} + (\nabla\cdot(Q_0\bm{u}-\bm{u}), \lambda \mathcal{Q}_h \nabla_g \cdot \bm{v})_{\mathcal{T}_h}\\
& + \langle R_b(Q_b\bm{u}-\bm{u}),\lambda (\mathcal{Q}_h \nabla_g \cdot \bm{v})\mathbf{n}\rangle_{\partial {\mathcal{T}_h}} + \langle R_b(\bm{u}-Q_0\bm{u}),\lambda (\mathcal{Q}_h \nabla_g \cdot \bm{v}) \mathbf{n}\rangle_{\partial {\mathcal{T}_h}}.
\end{split}
\end{equation}

Combining \eqref{weak divergence-1}-\eqref{weak divergence-2} with integration by parts, and using $\langle \lambda ({\nabla}\cdot\bm{u}\,{\rm I}) \mathbf{n},\bm{v}_b \rangle_{\partial \mathcal{T}_h}=0$, we obtain
\begin{equation}\label{error equ 4}
\begin{split}
&({\nabla}\cdot\bm{u},\lambda \mathcal{Q}_h \nabla_g \cdot \bm{v})_{\mathcal{T}_h} \\
= & (\lambda \mathcal{Q}_h {\nabla}\cdot\bm{u},\nabla_g \cdot \bm{v})_{\mathcal{T}_h}\\
= & (\lambda (\mathcal{Q}_h-I) {\nabla}\cdot\bm{u},\nabla \cdot \bm{v}_0)_{\mathcal{T}_h} + \langle R_b(\bm{v}_b-\bm{v}_0),\lambda (\mathcal{Q}_h {\nabla}\cdot\bm{u})\mathbf{n}\rangle_{\partial \mathcal{T}_h}\\
&-(\lambda \nabla\cdot({\nabla}\cdot\bm{u} \,{\rm I}),\bm{v}_0)_{\mathcal{T}_h} + \langle \lambda ({\nabla}\cdot\bm{u}\,{\rm I})\,\mathbf{n},(\bm{v}_0-\bm{v}_b) \rangle_{\partial \mathcal{T}_h} \\
= & -(\lambda \nabla\cdot({\nabla}\cdot\bm{u} \,{\rm I}),\bm{v}_0)_{\mathcal{T}_h} + (\lambda (\mathcal{Q}_h-I) {\nabla}\cdot\bm{u},\nabla \cdot \bm{v}_0)_{\mathcal{T}_h}\\
&+ \langle (\lambda {\nabla}\cdot\bm{u}\,{\rm I})\mathbf{n},(I-R_b)(\bm{v}_0-\bm{v}_b) \rangle_{\partial \mathcal{T}_h} \\
&+ \langle R_b(\bm{v}_0-\bm{v}_b),\lambda (I-\mathcal{Q}_h) ({\nabla}\cdot\bm{u})\,{\rm I}\,\mathbf{n}\rangle_{\partial \mathcal{T}_h}.
\end{split}
\end{equation}
Adding \eqref{error equ 2} and \eqref{error equ 4}, and recalling
$$
(\bm{f},\bm{v}_0)_{\mathcal{T}_h} = (2\mu \varepsilon_g(\bm{u}_h), \varepsilon_g (\bm{v}))_{\cT_h}+(\lambda \nabla_g\cdot \bm{u}_h, \nabla_g \cdot \bm{v})_{\cT_h}+s(Q_h \bm{u},\bm{v})-s(e_h,\bm{v}),
$$
we arrive at
\begin{equation}\label{error equ 5}
\begin{split}
&(\varepsilon(\bm{u}),2\mu \mathbb{Q}_h \varepsilon_g(\bm{v}))_{\mathcal{T}_h} + ({\nabla}\cdot\bm{u},\lambda \mathcal{Q}_h \nabla_g \cdot \bm{v})_{\mathcal{T}_h} \\
= & (\bm{f},\bm{v}_0)_{\mathcal{T}_h} + \langle R_b(\bm{v}_0-\bm{v}_b),(2\mu (I-\mathbb{Q}_h) \varepsilon (\bm{u}) + \lambda (I-\mathcal{Q}_h) ({\nabla}\cdot\bm{u})\,{\rm I} )\mathbf{n}\rangle_{\partial \mathcal{T}_h}\\
&+ \langle \sigma(\bm{u})\mathbf{n},(I-R_b)(\bm{v}_0-\bm{v}_b) \rangle_{\partial \mathcal{T}_h} + ( 2\mu (\mathbb{Q}_h-I) \varepsilon(\bm{u}),\varepsilon (\bm{v}_0))_{\mathcal{T}_h} \\
&+ (\lambda (\mathcal{Q}_h-I) {\nabla}\cdot\bm{u},\nabla \cdot \bm{v}_0)_{\mathcal{T}_h} \\
= & (2\mu \varepsilon_g(\bm{u}_h), \varepsilon_g (\bm{v}))_{\cT_h}+(\lambda \nabla_g\cdot \bm{u}_h, \nabla_g \cdot \bm{v})_{\cT_h}+s(Q_h \bm{u},\bm{v})-s(e_h,\bm{v}) \\
&+ \langle R_b(\bm{v}_0-\bm{v}_b),(2\mu (I-\mathbb{Q}_h) \varepsilon (\bm{u}) + \lambda (I-\mathcal{Q}_h) ({\nabla}\cdot\bm{u})\,{\rm I})\mathbf{n}\rangle_{\partial \mathcal{T}_h}\\
&+ \langle \sigma(\bm{u})\mathbf{n},(I-R_b)(\bm{v}_0-\bm{v}_b) \rangle_{\partial \mathcal{T}_h} + ( 2\mu (\mathbb{Q}_h-I) \varepsilon(\bm{u}),\varepsilon (\bm{v}_0))_{\mathcal{T}_h} \\
&+ (\lambda (\mathcal{Q}_h-I) {\nabla}\cdot\bm{u},\nabla \cdot \bm{v}_0)_{\mathcal{T}_h}. \\
\end{split}
\end{equation}

Finally, substituting \eqref{error equ 1} and \eqref{error equ 3} into \eqref{error equ 5}, and using the projection properties of $\mathbb{Q}_h$ and $\mathcal{Q}_h$ completes the proof of \eqref{error equ}.
\end{proof}

\smallskip
Assume that the following trace inequalities hold true for any $\boldsymbol{\psi} \in \nabla V^0(T) + G_1(T)$ and $\phi \in \nabla\cdot V^0(T) + G_2(T)$:
\begin{align}
\label{trace-inequality-1}\|\boldsymbol{\psi}\|_{\partial T}^2&\leq C h_T^{-1}\|\boldsymbol{\psi}\|_T^2,\\
\label{trace-inequality-2}\|\phi\|_{\partial T}^2&\leq C h_T^{-1}\|\phi\|_T^2.
\end{align}
%
Introduce a semi-norm in $V_h$ as follows
\begin{equation}\label{tbar norm}
\tbar \bm{v} \tbar:=\sqrt{a(\bm{v},\bm{v}) + s(\bm{v},\bm{v})}.
\end{equation}
It can be seen that this semi-norm is indeed a norm for $V_h^0$.

\begin{lemma}
For $\bm{v}\in V_h$, the following inequalities hold true
\begin{align}
\label{v0-1}\| \varepsilon(\bm{v}_0) \|_{\mathcal{T}_h} &\leq C\left(1+h^{\frac{-1-\gamma}{2}}\right) \tbar \bm{v} \tbar,\\ 
\label{v0-2}\| \nabla\cdot\bm{v}_0 \|_{\mathcal{T}_h} &\leq C\left(1+h^{\frac{-1-\gamma}{2}}\right) \tbar \bm{v} \tbar.
\end{align}
\end{lemma}
\begin{proof}
Using \eqref{weak gradient-2}, \eqref{weak divergence-2} and \eqref{trace-inequality-1}-\eqref{trace-inequality-2}, we obtain
\begin{equation}\label{v0-prof}
\begin{split}
\| \delta_{g1}\bm{v} \|_T &= \sup_{\boldsymbol{\psi}\in G_1(T)} \frac{\langle R_b(\bm{v}_b-\bm{v}_0),\boldsymbol{\psi}\mathbf{n}\rangle_{\partial T}}{\|\boldsymbol{\psi}\|_T} \leq Ch^{\frac{-1-\gamma}{2}} \tbar \bm{v} \tbar,\\
\| \delta_{g2}\bm{v} \|_T &= \sup_{\phi\in G_2(T)} \frac{\langle R_b(\bm{v}_b-\bm{v}_0),\phi\mathbf{n}\rangle_{\partial T}}{\| \phi \|_T} \leq C h^{\frac{-1-\gamma}{2}} \tbar \bm{v} \tbar.
\end{split}
\end{equation}
It follows from \eqref{weak gradient-1}, \eqref{weak divergence-1}, the triangle inequality, \eqref{tbar norm} and \eqref{v0-prof} that
\begin{equation*}
\begin{split}
\| \varepsilon(\bm{v}_0) \|_{\mathcal{T}_h} =& \left\| \frac{\nabla\bm{v}_0+\nabla\bm{v}_0^{\mathrm{T}}}{2} \right\|_{\mathcal{T}_h}\\
\leq & \left\| \frac{\nabla_g\bm{v}+\nabla_g\bm{v}^{\mathrm{T}}}{2} \right\|_{\mathcal{T}_h} + \left\| \frac{\delta_{g1}\bm{v}+\delta_{g1}\bm{v}^{\mathrm{T}}}{2} \right\|_{\mathcal{T}_h}\\
\leq & \| \varepsilon_g(\bm{v}) \|_{\mathcal{T}_h} + \| \delta_{g1}\bm{v} \|_{\mathcal{T}_h}\\
\leq & C\left(1+h^{\frac{-1-\gamma}{2}}\right) \tbar \bm{v} \tbar,\\
\\
\| \nabla\cdot\bm{v}_0 \|_{\mathcal{T}_h} \leq & \| \nabla_g\cdot\bm{v} \|_{\mathcal{T}_h} + \| \delta_{g2}\bm{v} \|_{\mathcal{T}_h}\\
\leq & C\left(1+h^{\frac{-1-\gamma}{2}}\right) \tbar \bm{v} \tbar.
\end{split}
\end{equation*}
\end{proof}

\begin{theorem}
Assume that the linear operator $R_b$ satisfies 
\begin{equation}\label{error estimate assumption}
\begin{split}
&\sup_{\bm{v}\in V_h} \frac{\langle \sigma(\bm{u})\mathbf{n},(I-R_b)(\bm{v}_0-\bm{v}_b) \rangle_{\partial \mathcal{T}_h}}{\tbar \bm{v} \tbar} \\
\leq & \ C h^{-\frac{\gamma}{2}} \left( \|(I-R_b)\nabla \bm{u}\|_{\partial \mathcal{T}_h} + \|(I-R_b)\nabla \cdot \bm{u}\|_{\partial \mathcal{T}_h} \right).
\end{split}
\end{equation}
Let $\bm{u}$ and $\bm{u}_h$ be the solutions of \eqref{eq:Problem} and \eqref{gWG_algorithm}, respectively. Then the following a priori bound holds:
\begin{equation}\label{l(u,v)}
\tbar\bm{e}_h\tbar \leq C R(\bm{u}),
\end{equation}
where
\begin{equation}
\begin{split}
R(\bm{u}) =& \left( 1+h^{\frac{-\gamma-1}{2}} \right) \left( \| (I-\mathbb{Q}_h) \nabla \bm{u} \|_{\mathcal{T}_h} + \| (I-\mathcal{Q}_h){\nabla} \cdot \bm{u} \|_{\mathcal{T}_h} \right) \\
&+ h^{-\frac{\gamma}{2}}( \|(I-\mathbb{Q}_h) \nabla \bm{u}\|_{\partial \mathcal{T}_h}+\|(I-\mathcal{Q}_h) ({\nabla}\cdot\bm{u})\|_{\partial \mathcal{T}_h} ) \\
&+ \|\nabla(Q_0\bm{u}-\bm{u}))\|_{\mathcal{T}_h} + \|\nabla\cdot(Q_0\bm{u}-\bm{u})\|_{\mathcal{T}_h} \\
&+ \left(h^{-\frac{1}{2}} + h^{\frac{\gamma}{2}}\right) \left(\|R_b(Q_b\bm{u}-\bm{u})\|_{\partial {\mathcal{T}_h}} + \| R_b(\bm{u}-Q_0\bm{u}) \|_{\partial {\mathcal{T}_h}} \right)\\
&+ h^{-\frac{\gamma}{2}} \left( \|(I-R_b)\nabla \bm{u}\|_{\partial \mathcal{T}_h} + \|(I-R_b)\nabla \cdot \bm{u}\|_{\partial \mathcal{T}_h} \right).
\end{split}
\end{equation}
\end{theorem}
\begin{proof}
Setting $\bm{v}=\bm{e}_h$ in \eqref{error equ} yields $\tbar \bm{e}_h \tbar^2 = |\ell (\bm{u},\bm{e}_h)|$. We now estimate each term appearing in $|\ell (\bm{u},\bm{e}_h)|$. Throughout, the Cauchy-Schwarz inequality and the triangule inequality are repeatedly employed. 

Using the stability and projection properties of $\mathbb{Q}_h$ and $Q_0$, we obtain 
\begin{equation}\label{error estimate 1}
\begin{split}
&|(\varepsilon(Q_0 \bm{u}),2\mu (I-\mathbb{Q}_h) \varepsilon_g(\bm{e}_h))_{\mathcal{T}_h}| \\
=& |(\varepsilon(Q_0 \bm{u}),2\mu (I-\mathbb{Q}_h) \varepsilon_g(\bm{e}_h))_{\mathcal{T}_h} - (\mathbb{Q}_h\varepsilon(Q_0 \bm{u}),2\mu (I-\mathbb{Q}_h) \varepsilon_g(\bm{e}_h))_{\mathcal{T}_h}| \\
=& |((I-\mathbb{Q}_h)\varepsilon(Q_0 \bm{u}),2\mu (I-\mathbb{Q}_h) \varepsilon_g(\bm{e}_h))_{\mathcal{T}_h}| \\
=& |((I-\mathbb{Q}_h)\varepsilon(Q_0 \bm{u}),2\mu \varepsilon_g(\bm{e}_h))_{\mathcal{T}_h}| \\
\leq& C\| (I-\mathbb{Q}_h)\varepsilon(Q_0 \bm{u}) \|_{\mathcal{T}_h} \, \|\varepsilon_g(\bm{e}_h)\|_{\mathcal{T}_h} \\
\leq& C\| (I-\mathbb{Q}_h)(\nabla Q_0 \bm{u}) \|_{\mathcal{T}_h} \, \|\varepsilon_g(\bm{e}_h)\|_{\mathcal{T}_h} \\
\leq& C\| (I-\mathbb{Q}_h)(\nabla \bm{u}) \|_{\mathcal{T}_h} \, \tbar\bm{e}_h\tbar. \\
\end{split}
\end{equation}
Similarly,
\begin{equation}\label{error estimate 2}
\begin{split}
|({\nabla} \cdot Q_0\bm{u},\lambda (I-\mathcal{Q}_h) \nabla_g \cdot \bm{e}_h)_{\mathcal{T}_h}| =& |((I-\mathcal{Q}_h) {\nabla} \cdot Q_0\bm{u},\lambda \nabla_g \cdot \bm{e}_h)_{\mathcal{T}_h}| \\
\leq& C \| (I-\mathcal{Q}_h){\nabla} \cdot Q_0\bm{u} \|_{\mathcal{T}_h} \, \|\nabla_g \cdot \bm{e}_h\|_{\mathcal{T}_h} \\
\leq& C \| (I-\mathcal{Q}_h){\nabla} \cdot \bm{u} \|_{\mathcal{T}_h} \, \tbar\bm{e}_h\tbar. \\
\end{split}
\end{equation}
From \eqref{tbar norm} we have
\begin{equation}\label{error estimate 3}
\begin{split}
|s(Q_h \bm{u},\bm{e}_h)| =& \left|\sum_{T\in \cT_h}\rho h_T^{\gamma}\langle R_b(Q_0 \bm{u}-Q_b \bm{u}),R_b(\bm{e}_0-\bm{e}_b) \rangle_{\partial T}\right| \\
\leq& Ch^{\frac{\gamma}{2}} \|R_b(Q_0 \bm{u}-Q_b \bm{u})\|_{\partial \mathcal{T}_h}\,h^{\frac{\gamma}{2}} \|R_b(\bm{e}_0-\bm{e}_b)\|_{\partial \mathcal{T}_h}\\
\leq& Ch^{\frac{\gamma}{2}} \left( \|R_b(Q_0 \bm{u}-\bm{u})\|_{\partial \mathcal{T}_h}+\|R_b(\bm{u}-Q_b \bm{u})\|_{\partial \mathcal{T}_h} \right) \, \tbar\bm{e}_h\tbar. \\
\end{split}
\end{equation}
Similarly,
\begin{equation}\label{error estimate 4}
\begin{split}
&|\langle R_b(\bm{e}_0-\bm{e}_b),(2\mu (I-\mathbb{Q}_h) \varepsilon (\bm{u}) + \lambda (I-\mathcal{Q}_h) ({\nabla}\cdot\bm{u})\,{\rm I})\mathbf{n}\rangle_{\partial \mathcal{T}_h}| \\
\leq& C\| R_b(\bm{e}_0-\bm{e}_b) \|_{\partial \mathcal{T}_h} \, ( \|(I-\mathbb{Q}_h) \varepsilon (\bm{u})\|_{\partial \mathcal{T}_h}+\|(I-\mathcal{Q}_h) ({\nabla}\cdot\bm{u})\|_{\partial \mathcal{T}_h} )\\
\leq& Ch^{-\frac{\gamma}{2}}( \|(I-\mathbb{Q}_h) \nabla \bm{u}\|_{\partial \mathcal{T}_h}+\|(I-\mathcal{Q}_h) ({\nabla}\cdot\bm{u})\|_{\partial \mathcal{T}_h} )\, \tbar \bm{e}_h \tbar.\\
\end{split}
\end{equation}

Next, using \eqref{v0-1} and \eqref{v0-2},
\begin{equation}\label{error estimate 5}
\begin{split}
&|( 2\mu (\mathbb{Q}_h-I) \varepsilon(\bm{u}),\varepsilon (\bm{e}_0))_{\mathcal{T}_h} + (\lambda (\mathcal{Q}_h-I) {\nabla}\cdot\bm{u},\nabla \cdot \bm{e}_0)_{\mathcal{T}_h}| \\
\leq& C \left(\|(\mathbb{Q}_h-I) \varepsilon(\bm{u})\|_{\mathcal{T}_h} \, \|\varepsilon (\bm{e}_0)\|_{\mathcal{T}_h} + \| (\mathcal{Q}_h-I) {\nabla}\cdot\bm{u} \|_{\mathcal{T}_h} \, \| \nabla \cdot \bm{e}_0 \|_{\mathcal{T}_h} \right) \\
\leq& C \left(1+h^{\frac{-1-\gamma}{2}}\right)\left( \|(\mathbb{Q}_h-I) \nabla \bm{u} \|_{\mathcal{T}_h} + \|(\mathcal{Q}_h-I) {\nabla}\cdot\bm{u}\|_{\mathcal{T}_h} \right) \, \tbar \bm{e}_h \tbar. \\
\end{split}
\end{equation}
From projection properties,
\begin{equation}\label{error estimate 6}
\begin{split}
&|(\varepsilon(Q_0\bm{u}-\bm{u}), 2\mu \mathbb{Q}_h \varepsilon_g(\bm{e}_h))_{\mathcal{T}_h} + (\nabla\cdot(Q_0\bm{u}-\bm{u}), \lambda \mathcal{Q}_h \nabla_g \cdot \bm{e}_h)_{\mathcal{T}_h}|\\
\leq& C \left( \|\varepsilon(Q_0\bm{u}-\bm{u})\|_{\mathcal{T}_h} \, \|\mathbb{Q}_h \varepsilon_g(\bm{e}_h)\|_{\mathcal{T}_h} + \|\nabla\cdot(Q_0\bm{u}-\bm{u})\|_{\mathcal{T}_h} \, \| \mathcal{Q}_h \nabla_g \cdot \bm{e}_h \|_{\mathcal{T}_h} \right) \\
\leq& C \left( \|\nabla(Q_0\bm{u}-\bm{u})\|_{\mathcal{T}_h} \, \|\varepsilon_g(\bm{e}_h)\|_{\mathcal{T}_h} + \|\nabla\cdot(Q_0\bm{u}-\bm{u})\|_{\mathcal{T}_h} \, \| \nabla_g \cdot \bm{e}_h \|_{\mathcal{T}_h} \right) \\
\leq& C \left( \|\nabla(Q_0\bm{u}-\bm{u})\|_{\mathcal{T}_h} + \|\nabla\cdot(Q_0\bm{u}-\bm{u})\|_{\mathcal{T}_h} \right) \, \tbar \bm{e}_h \tbar. \\
\end{split}
\end{equation}
Finally, using the trace inequalities \eqref{trace-inequality-1}--\eqref{trace-inequality-2},
\begin{equation}\label{error estimate 7}
\begin{split}
&|\langle R_b(Q_b\bm{u}-\bm{u})+R_b(\bm{u}-Q_0\bm{u}),(2\mu \mathbb{Q}_h \varepsilon_g(\bm{e}_h) + \lambda \mathcal{Q}_h \nabla_g \cdot \bm{e}_h \,{\rm I})\mathbf{n}\rangle_{\partial {\mathcal{T}_h}}| \\
\leq& C \left(\|R_b(Q_b\bm{u}-\bm{u})\|_{\partial {\mathcal{T}_h}} + \| R_b(\bm{u}-Q_0\bm{u}) \|_{\partial {\mathcal{T}_h}} \right) \, \left( \|\mathbb{Q}_h \varepsilon_g(\bm{e}_h)\|_{\partial {\mathcal{T}_h}}+\|\mathcal{Q}_h \nabla_g \cdot \bm{e}_h\|_{\partial {\mathcal{T}_h}} \right) \\
\leq& C \left(\|R_b(Q_b\bm{u}-\bm{u})\|_{\partial {\mathcal{T}_h}} + \| R_b(\bm{u}-Q_0\bm{u}) \|_{\partial {\mathcal{T}_h}} \right) \, \left( \|\varepsilon_g(\bm{e}_h)\|_{\partial {\mathcal{T}_h}}+\|\nabla_g \cdot \bm{e}_h\|_{\partial {\mathcal{T}_h}} \right) \\
\leq& C h^{-\frac{1}{2}} \left(\|R_b(Q_b\bm{u}-\bm{u})\|_{\partial {\mathcal{T}_h}} + \| R_b(\bm{u}-Q_0\bm{u}) \|_{\partial {\mathcal{T}_h}} \right) \, \tbar \bm{e}_h \tbar. \\
\end{split}
\end{equation}
Combining \eqref{error estimate 1}--\eqref{error estimate 7} with \eqref{error estimate assumption} yields \eqref{l(u,v)}, and the proof is complete.
\end{proof}

The estimate \eqref{l(u,v)} quantifies how the discretization error is governed by the approximation properties of the projection operators $\mathbb{Q}_h$, $\mathcal{Q}_h$, and $Q_0$, together with the stabilization operator $R_b$. More precisely, $R(\bm{u})$ collects the contributions from element interior projection errors, boundary projection errors, and local stabilization mismatch along inter-element boundaries. The parameter $\gamma$ controls the relative strength of the boundary stabilization, and hence determines how the interior and boundary projection terms scale with respect to $h$. In particular, the assumption \eqref{error estimate assumption} ensures that the boundary residual produced by $R_b$ can be bounded in terms of computable projection defects of the exact solution. As a result, \eqref{l(u,v)} shows that the discrete error is fully dominated by these projection errors, which is the key mechanism enabling optimal-order convergence once appropriate approximation estimates are invoked.

\section{Numerical Examples}\label{EX}
\smallskip
In this section, we present two numerical examples to demonstrate the performance of the proposed gWG algorithm under different mesh partitions and choices of possible non-polynomial approximation spaces. Throughout the computation, we set $\rho=1$, take the spatial dimension $d=2$, and consider the computational domain $\Omega=(0,1)^{2}$. The Dirichlet boundary data and the forcing term $\bm{f}$ are chosen so that the exact solution is the selected $\bm{u}$. The linear operator $R_b$ is selected as either $Q_b$ or the identity operator $I$. The gWG finite element space for linear elasticity has the general form
$$
V^0(T)|V^b(\partial T)|G_1(T)|G_2(T).
$$ 
We will examine several combinations of the following local spaces:
\begin{equation*}
\begin{split}
&V^0(T)=\text{span}\left\{ \left[\begin{matrix} 1\\0 \end{matrix}\right],\,\left[\begin{matrix} 0\\ 1 \end{matrix}\right],\, \{\bm{\sigma}(t_{i})\}_{i=1}^{p} \right\} \text{ or } [P_1(T)]^2,\\
&V^b(\partial T)=[P_k(\partial T)]^{2} \, (k=0,1) \text{ or } \text{RM}(\partial T),\\
&G_1(T)=[P_0(T)]^{2\times2},\,G_2(T)=P_0(T).
\end{split}
\end{equation*}
Here the random variable $t_{i}$ on each element $T$ is defined by $t_i=\bm{w}_{0i}\cdot(X-X_{0i})$, where 
$$
\bm{w}_{0i}=\left[\begin{matrix} \text{rand}(1)-0.5\\ \text{rand}(1)-0.5 \end{matrix}\right], \ \ X=(x,y)^T,
$$ 
and rand(1) is the order of generating random number in MATLAB between $0$ and $1$. The random point $X_{0i}$ will be illustrated later on. The selected activation-based basis functions are described below:
\begin{align}
& \{\mbox{sin}~\bm{\sigma}(t_i)\}_{i=1}^{4}=\left\{\left[\begin{matrix}\sin(t_1)\\0\end{matrix}\right],\,\left[\begin{matrix}0\\\sin(t_2)\end{matrix}\right],\,\left[\begin{matrix}\sin(t_3)\\0\end{matrix}\right],\,\left[\begin{matrix}0\\\sin(t_4)\end{matrix}\right]\right\},\nonumber\\
& \{\mbox{cos}~\bm{\sigma}(t_i)\}_{i=1}^{4}=\left\{\left[\begin{matrix}\cos(t_1)\\0\end{matrix}\right],\,\left[\begin{matrix}0\\\cos(t_2)\end{matrix}\right],\,\left[\begin{matrix}\cos(t_3)\\0\end{matrix}\right],\,\left[\begin{matrix}0\\\cos(t_4)\end{matrix}\right]\right\},\nonumber\\
& \{\mbox{Sigmoid}~\bm{\sigma}(t_i)\}_{i=1}^{4}=\left\{\left[\begin{matrix}\frac{1}{1+e^{-t_1}}\\0\end{matrix}\right],\,\left[\begin{matrix}0\\\frac{1}{1+e^{-t_2}}\end{matrix}\right],\,\left[\begin{matrix}\frac{1}{1+e^{-t_3}}\\0\end{matrix}\right],\,\left[\begin{matrix}0\\\frac{1}{1+e^{-t_4}}\end{matrix}\right]\right\},\nonumber\\
& \{\mbox{ReLU}~\bm{\sigma}(t_i)\}_{i=1}^{4}=\left\{\left[\begin{matrix}\max\{0,t_1\}\\0\end{matrix}\right],\,\left[\begin{matrix}0\\\max\{0,t_2\}\end{matrix}\right],\,\left[\begin{matrix}\max\{0,t_3\}\\0\end{matrix}\right],\,\left[\begin{matrix}0\\\max\{0,t_4\}\end{matrix}\right]\right\},\nonumber\\
& \{\mbox{L-ReLU}~\bm{\sigma}(t_i)\}_{i=1}^{4}=\left\{\left[\begin{matrix}\mbox{Leaky-ReLU}~\sigma(t_1)\\0\end{matrix}\right],\,\left[\begin{matrix}0\\\mbox{Leaky-ReLU}~\sigma(t_2)\end{matrix}\right]\right.,\,\nonumber\\
&\qquad\qquad\qquad\qquad\qquad\qquad\qquad\quad\left.\left[\begin{matrix}\mbox{Leaky-ReLU}~\sigma(t_3)\\0\end{matrix}\right],\,\left[\begin{matrix}0\\\mbox{Leaky-ReLU}~\sigma(t_4)\end{matrix}\right]\right\},\nonumber
\end{align}
where 
$$
\mbox{Leaky-ReLU}~\sigma(t)=\begin{cases} t& \mbox{if}~ t>0,\\ \varepsilon t& \mbox{otherwise.}\end{cases}
$$

\smallskip
\noindent
{\bf Uniform rectangular partition:} For each rectangular element, the vertices $A_j=[x_{j};\,y_{j}]$ $(j=1,\ldots,4)$ are ordered counterclockwise starting from the lower-left corner. The random points $X_{0i}$ are generated by
$$
X_{0i}=\left[\begin{matrix}\text{rand}(1)\cdot(x_{2}-x_{1})+x_{1}\\\text{rand}(1)\cdot(y_{3}-y_{2})+y_{2}\end{matrix}\right], \qquad (i=1,\ldots,p).
$$

\smallskip
\noindent
{\bf Uniform triangular partition:} The finite element partition consists of triangular elements generated using \texttt{delaunayTriangulation} in MATLAB. For each element $T$, denote its three vertices by $A_j$ $(j=1,2,3)$. The random points are generated as
$$
X_{0i}=\beta_{0i}\cdot(\alpha_{0i} A_1+(1-\alpha_{0i}) A_2) + (1-\beta_{0i}) A_3, \qquad (i=1,\ldots,p),
$$
where $\alpha_{0i}=\text{rand}(1)$ and $\beta_{0i}=\sqrt{\text{rand}(1)}$.

\smallskip
The discrete $L_2$ and $L_{\infty}$ errors are defined by
\begin{align*}
\|\bm{u}-\bm{u}_0\|&=\left(\sum_{T\in\mathcal{T}_h}\int_T |\bm{u}-\bm{u}_0|^2 dT\right)^{\frac{1}{2}},\\
\|\bm{u}-\bm{u}_b\|&=\left(\sum_{e\in\mathcal{E}_h}h_e\int_e |\bm{u}-\bm{u}_b|^2 ds\right)^{\frac{1}{2}},\\
\|\bm{u}-\bm{u}_0\|_{\infty}&=\max_{{T\in\mathcal{T}_h}}\left|{u}_i(x_c,y_c)-{u}_{0,i}(x_c,y_c)\right|, \quad i=1,2,\\
\|\bm{u}-\bm{u}_b\|_{\infty}&=\max_{{e\in\mathcal{E}_h}}|{u}_i(x_{m},y_{m})-{u}_{b,i}(x_{m},y_{m})|,\quad i=1,2,
\end{align*}
where $(x_c,y_c)$ and $(x_{m},y_{m})$ denote the barycenter of element $T$ and the midpoint of edge $e$, respectively.

\begin{exmp}\label{general EX}
Consider the exact solution 
$$
\bm{u}=\left[\begin{matrix}\sin(x) \sin(y) \\ 1\end{matrix}\right].
$$
The Lam\'{e} constants are chosen as $\mu=0.5$ and $\lambda=1$. Tables \ref{general_EX_Triangle}--\ref{general_EX_Rectangle_2} report the numerical errors and convergence rates for various combinations of mesh type, $V^0(T)$, $V^b(\partial T)$, $R_b$, and $\gamma$. 
\end{exmp}


\smallskip
Table \ref{general_EX_Triangle} shows that for triangular meshes, the gWG method achieves the expected convergence behavior when $V^0(T)$ includes activation-based enrichments and $V^b(\partial T)=\text{RM}(\partial T)$. In particular, the spaces $V^0_2(T)$, $V^0_4(T)$, and $V^0_6(T)$ all exhibit similar convergence properties under this choice of $V^b(\partial T)$. In contrast, when $V^b(\partial T)=[P_0(\partial T)]^2$ is employed on triangular elements, the results display a divergent tendency. 

\begin{table}[H]
\caption{Errors and convergence rates for gWG space $V^0(T)|V^b(\partial T)|[P_0(T)]^{2\times2}|P_0(T)$ on triangular elements for Example \ref{general EX}; $R_b=Q_b$; $\gamma=-1$; $\mu=0.5$; $\lambda=1$. }\label{general_EX_Triangle}
\small
\centering
\begin{tabular}{|c c c c c c c c c|}
\hline
$h$&$\|\bm{u}-\bm{u}_0\|$&Rate&$\|\bm{u}-\bm{u}_b\|$&Rate&$\|\bm{u}-\bm{u}_0\|_{\infty}$&Rate&$\|\bm{u}-\bm{u}_b\|_{\infty}$&Rate\\
\hline
\multicolumn{9}{|c|}{$V^0_1(T)=[P_1(T)]^2$,~$V^b_1(\partial T)=[P_1(\partial T)]^2$}\\
\hline
1/8&4.86E-03&1.98 &3.44E-03&1.91 &5.47E-03&1.94 &2.85E-03&1.40 \\
1/16&1.22E-03&1.99 &8.94E-04&1.94 &1.40E-03&1.97 &9.07E-04&1.65 \\
1/32&3.06E-04&2.00 &2.27E-04&1.98 &3.61E-04&1.95 &2.63E-04&1.78 \\
1/64&7.66E-05&2.00 &5.70E-05&1.99 &9.41E-05&1.94 &7.14E-05&1.88 \\
\hline
\multicolumn{9}{|c|}{$V^0_2(T)=[P_1(T)]^2$,~$V^b_2(\partial T)=\text{RM}(\partial T)$}\\
\hline
1/8&4.90E-03&1.98 &2.24E-02&1.09 &5.48E-03&1.95 &2.92E-03&1.42 \\
1/16&1.23E-03&1.99 &1.09E-02&1.04 &1.40E-03&1.97 &9.19E-04&1.67 \\
1/32&3.08E-04&2.00 &5.36E-03&1.02 &3.63E-04&1.94 &2.66E-04&1.79 \\
1/64&7.71E-05&2.00 &2.66E-03&1.01 &9.45E-05&1.94 &7.22E-05&1.88 \\
\hline
\multicolumn{9}{|c|}{$V^0_3(T)=[P_1(T)]^2$,~$V^b_3(\partial T)=[P_0(\partial T)]^2$}\\
\hline
1/8&1.01E-02&1.11 &3.49E-02&0.95 &9.92E-03&1.16 &1.76E-02&0.04 \\
1/16&8.59E-03&0.24 &2.17E-02&0.68 &6.69E-03&0.57 &1.75E-02&0.01 \\
1/32&8.41E-03&0.03 &1.72E-02&0.33 &7.18E-03&-0.10 &1.75E-02&0.00 \\
1/64&8.38E-03&0.01 &1.60E-02&0.11 &7.40E-03&-0.04 &1.75E-02&0.00 \\
\hline
\multicolumn{9}{|c|}{$V^0_4(T)=\text{span}\{ [1;0],\,[0;1],\,\{\mbox{sin}~\bm{\sigma}(t_i)\}_{i=1}^{4} \}$,~$V^b_4(\partial T)=RM(\partial T)$}\\
\hline
1/8&4.90E-03&1.98 &2.24E-02&1.09 &5.48E-03&1.95 &2.92E-03&1.42 \\
1/16&1.23E-03&1.99 &1.09E-02&1.04 &1.40E-03&1.97 &9.19E-04&1.67 \\
1/32&3.08E-04&2.00 &5.36E-03&1.02 &3.64E-04&1.94 &2.66E-04&1.79 \\
1/64&7.71E-05&2.00 &2.66E-03&1.01 &9.45E-05&1.95 &7.22E-05&1.88 \\
\hline
\multicolumn{9}{|c|}{$V^0_5(T)=\text{span}\{ [1;0],\,[0;1],\,\{\mbox{sin}~\bm{\sigma}(t_i)\}_{i=1}^{4} \}$,~$V^b_5(\partial T)=[P_0(\partial T)]^2$}\\
\hline
1/8&1.00E-02&1.12 &3.48E-02&0.96 &9.84E-03&1.17 &1.74E-02&0.06 \\
1/16&8.59E-03&0.22 &2.17E-02&0.68 &6.69E-03&0.56 &1.75E-02&-0.01 \\
1/32&8.41E-03&0.03 &1.72E-02&0.33 &7.18E-03&-0.10 &1.74E-02&0.00 \\
1/64&8.37E-03&0.01 &1.59E-02&0.11 &7.40E-03&-0.04 &1.75E-02&0.00 \\
\hline
\multicolumn{9}{|c|}{$V^0_6(T)=\text{span}\{ [1;0],\,[0;1],\,\{\mbox{Sigmoid}~\bm{\sigma}(t_i)\}_{i=1}^{4} \}$,~$V^b_6(\partial T)=\text{RM}(\partial T)$}\\
\hline
1/8&4.90E-03&1.98 &2.24E-02&1.09 &5.48E-03&1.95 &2.91E-03&1.44 \\
1/16&1.23E-03&1.99 &1.09E-02&1.04 &1.40E-03&1.97 &9.19E-04&1.67 \\
1/32&3.09E-04&2.00 &5.36E-03&1.02 &5.59E-04&1.32 &2.66E-04&1.79 \\
1/64&7.71E-05&2.00 &2.66E-03&1.01 &9.45E-05&2.57 &7.22E-05&1.88 \\
\hline
\multicolumn{9}{|c|}{$V^0_7(T)=\text{span}\{ [1;0],\,[0;1],\,\{\mbox{Sigmoid}~\bm{\sigma}(t_i)\}_{i=1}^{4} \}$,~$V^b_7(\partial T)=[P_0(\partial T)]^2$}\\
\hline
1/8&1.01E-02&1.11 &3.49E-02&0.95 &9.90E-03&1.16 &1.76E-02&0.04 \\
1/16&8.59E-03&0.23 &2.17E-02&0.68 &6.69E-03&0.57 &1.75E-02&0.01 \\
1/32&8.41E-03&0.03 &1.72E-02&0.33 &7.17E-03&-0.10 &1.74E-02&0.00 \\
1/64&8.37E-03&0.01 &1.59E-02&0.12 &7.67E-03&-0.10 &2.17E-02&-0.32 \\
\hline
\end{tabular}
\end{table}

\vspace{7cm}

\smallskip
Table \ref{general_EX_Rectangle_1} presents the results on uniform rectangular partitions. For the activation-based spaces $V^0_3(T)$ and $V^0_5(T)$, as well as the polynomial space $V^0_2(T)$, the gWG method exhibits similar convergence behavoir when $\gamma=-1$. For the choice $V^0_4(T)$ with $\gamma=0$, we observe convergence of approximately first order in both $\|u-u_0\|$ and $\|u-u_b\|$. A comparison between Table \ref{general_EX_Triangle} and Table \ref{general_EX_Rectangle_1} indicates that the convergence performance depends not only on the selection of $V^0(T)$ and the stabilization parameter $\gamma$, but also on the underlying mesh partition.

\begin{table}[H]
\caption{Errors and convergence rates for gWG space $V^0(T)|V^b(\partial T)|[P_0(T)]^{2\times2}|P_0(T)$ on rectangular elements for Example \ref{general EX}; $R_b=Q_b$; $\mu=0.5$; $\lambda=1$.}\label{general_EX_Rectangle_1}
\small
\centering
\begin{tabular}{|c c c c c c c c c|}
\hline
$h$&$\|\bm{u}-\bm{u}_0\|$&Rate&$\|\bm{u}-\bm{u}_b\|$&Rate&$\|\bm{u}-\bm{u}_0\|_{\infty}$&Rate&$\|\bm{u}-\bm{u}_b\|_{\infty}$&Rate\\
\hline
\multicolumn{9}{|c|}{$V^0_1(T)=[P_1(T)]^2$,~$V^b_1(\partial T)=\text{RM}(\partial T)$~and~$\gamma=-1$}\\
\hline
1/8&8.36E-03&1.94 &1.80E-02&1.16 &1.00E-02&1.86 &7.46E-03&1.43 \\
1/16&2.12E-03&1.98 &8.46E-03&1.09 &2.82E-03&1.83 &2.22E-03&1.75 \\
1/32&5.33E-04&1.99 &4.11E-03&1.04 &7.60E-04&1.89 &6.10E-04&1.86 \\
1/64&1.33E-04&2.00 &2.03E-03&1.02 &1.99E-04&1.93 &1.61E-04&1.92 \\
\hline
\multicolumn{9}{|c|}{$V^0_2(T)=[P_1(T)]^2$,~$V^b_2(\partial T)=[P_0(\partial T)]^2$~and~$\gamma=-1$}\\
\hline
1/8&8.38E-03&1.95 &2.50E-02&1.13 &1.00E-02&1.87 &7.53E-03&1.57 \\
1/16&2.12E-03&1.98 &1.19E-02&1.07 &2.83E-03&1.83 &2.24E-03&1.75 \\
1/32&5.33E-04&1.99 &5.81E-03&1.03 &7.61E-04&1.89 &6.13E-04&1.87 \\
1/64&1.33E-04&2.00 &2.87E-03&1.02 &1.99E-04&1.93 &1.61E-04&1.93 \\
\hline
\multicolumn{9}{|c|}{$V^0_3(T)=\text{span}\{ [1;0],\,[0;1],\,\{\mbox{sin}~\bm{\sigma}(t_i)\}_{i=1}^{4} \}$,~$V^b_3(\partial T)=[P_0(\partial T)]^2$~and~$\gamma=-1$}\\
\hline
1/8&8.38E-03&1.95 &2.50E-02&1.13 &1.00E-02&1.87 &7.52E-03&1.57 \\
1/16&2.12E-03&1.98 &1.19E-02&1.07 &2.82E-03&1.83 &2.24E-03&1.75 \\
1/32&5.33E-04&1.99 &5.81E-03&1.03 &7.61E-04&1.89 &6.14E-04&1.86 \\
1/64&1.34E-04&2.00 &2.87E-03&1.02 &1.99E-04&1.93 &1.61E-04&1.93 \\
\hline
\multicolumn{9}{|c|}{$V^0_4(T)=\text{span}\{ [1;0],\,[0;1],\,\{\mbox{cos}~\bm{\sigma}(t_i)\}_{i=1}^{4} \}$,~$V^b_4(\partial T)=[P_0(\partial T)]^2$~and~$\gamma=0$}\\
\hline
1/8&4.95E-02&0.94 &2.94E-02&0.99 &5.46E-02&0.93 &2.52E-02&0.63 \\
1/16&2.54E-02&0.96 &1.58E-02&0.90 &3.08E-02&0.83 &1.72E-02&0.55 \\
1/32&1.29E-02&0.97 &8.36E-03&0.92 &1.71E-02&0.85 &1.23E-02&0.48 \\
1/64&6.55E-03&0.98 &4.36E-03&0.94 &9.32E-03&0.87 &6.42E-03&0.94 \\
\hline
\multicolumn{9}{|c|}{$V^0_5(T)=\text{span}\{ [1;0],\,[0;1],\,\{\mbox{Sigmoid}~\bm{\sigma}(t_i)\}_{i=1}^{4} \}$,~$V^b_5(\partial T)=[P_0(\partial T)]^2$~and~$\gamma=-1$}\\
\hline
1/8&8.39E-03&1.95 &2.50E-02&1.13 &1.00E-02&1.86 &7.54E-03&1.56 \\
1/16&2.12E-03&1.98 &1.19E-02&1.07 &2.83E-03&1.83 &2.24E-03&1.75 \\
1/32&5.33E-04&1.99 &5.81E-03&1.03 &7.60E-04&1.89 &6.14E-04&1.87 \\
1/64&1.33E-04&2.00 &2.87E-03&1.02 &1.99E-04&1.93 &1.61E-04&1.93 \\
\hline
\end{tabular}
\end{table}

\smallskip
Table \ref{general_EX_Rectangle_2} illustrates the performance of the Leaky-ReLU and ReLU activation functions used in $V^0(T)$ on rectangular elements. It can be observed that the space $V^0_1(T)$, which extends the classical $P_1$ non-conforming element, achieves convergence rates of $O(h^2)$ for $\|u-u_0\|$ and $O(h)$ for $\|u-u_b\|$. For the space $V^0_2(T)$ with $\gamma=0$, the expected convergence rates are obtained, where the approximate functions are piecewise continuous within each element. In contrast, the ReLU activation functions in spaces $V^0_3(T)$ through $V^0_5(T)$ exhibit performance with uncertainty.

\begin{table}[H]
\caption{Errors and convergence rates for gWG space $V^0(T)|[P_0(\partial T)]^2|[P_0(T)]^{2\times2}|P_0(T)$ on rectangular elements for Example \ref{general EX}; $R_b=Q_b$; $\mu=0.5$; $\lambda=1$. }\label{general_EX_Rectangle_2}
\small
\centering
\begin{tabular}{|c c c c c c c c c|}
\hline
$h$&$\|\bm{u}-\bm{u}_0\|$&Rate&$\|\bm{u}-\bm{u}_b\|$&Rate&$\|\bm{u}-\bm{u}_0\|_{\infty}$&Rate&$\|\bm{u}-\bm{u}_b\|_{\infty}$&Rate\\
\hline
\multicolumn{9}{|c|}{$V^0_1(T)=\text{span}\{ [1;0],\,[0;1],\,\{\mbox{L-ReLU}~\bm{\sigma}(t_i)\}_{i=1}^{4} \}$,~$\gamma=-1$~and~$\varepsilon=1$}\\
\hline
1/8&8.38E-03&1.95 &2.50E-02&1.13 &1.00E-02&1.87 &7.53E-03&1.57 \\
1/16&2.12E-03&1.98 &1.19E-02&1.07 &2.83E-03&1.83 &2.24E-03&1.75 \\
1/32&5.33E-04&1.99 &5.81E-03&1.03 &7.61E-04&1.89 &6.13E-04&1.87 \\
1/64&1.33E-04&2.00 &2.87E-03&1.02 &1.99E-04&1.93 &1.61E-04&1.93 \\
\hline
\multicolumn{9}{|c|}{$V^0_2(T)=\text{span}\{ [1;0],\,[0;1],\,\{\mbox{L-ReLU}~\bm{\sigma}(t_i)\}_{i=1}^{4} \}$,~$\gamma=0$~and~$\varepsilon=1/10$}\\
\hline
1/8&4.89E-02&0.93 &2.98E-02&0.98 &5.37E-02&0.89 &2.71E-02&0.44 \\
1/16&2.52E-02&0.96 &1.58E-02&0.92 &3.10E-02&0.79 &1.98E-02&0.46 \\
1/32&1.29E-02&0.96 &8.36E-03&0.92 &1.72E-02&0.85 &1.10E-02&0.85 \\
1/64&6.56E-03&0.98 &4.37E-03&0.94 &9.13E-03&0.91 &6.62E-03&0.73 \\
\hline
\multicolumn{9}{|c|}{$V^0_3(T)=\text{span}\{ [1;0],\,[0;1],\,\{\mbox{ReLU}~\bm{\sigma}(t_i)\}_{i=1}^{4} \}$~and~$\gamma=-1$}\\
\hline
1/8&1.89E-02&1.32 &2.67E-02&1.03 &1.55E-02&1.48 &2.23E-02&-0.13 \\
1/16&9.46E-03&1.00 &1.32E-02&1.02 &9.11E-03&0.76 &1.30E-02&0.78 \\
1/32&5.33E-03&0.83 &7.52E-03&0.81 &9.22E-03&-0.02 &1.11E-02&0.24 \\
1/64&3.84E-03&0.47 &5.42E-03&0.47 &2.18E-02&-1.24 &3.13E-02&-1.50 \\
\hline
\multicolumn{9}{|c|}{$V^0_4(T)=\text{span}\{ [1;0],\,[0;1],\,\{\mbox{ReLU}~\bm{\sigma}(t_i)\}_{i=1}^{4} \}$~and~$\gamma=0$}\\
\hline
1/8&4.89E-02&0.95 &2.96E-02&0.95 &5.33E-02&0.90 &2.48E-02&0.24 \\
1/16&2.54E-02&0.95 &1.57E-02&0.91 &3.17E-02&0.75 &1.82E-02&0.45 \\
1/32&1.29E-02&0.97 &8.35E-03&0.91 &1.71E-02&0.89 &1.09E-02&0.74 \\
1/64&9.87E-03&0.39 &1.23E-02&-0.55 &1.06E-01&-2.63 &2.61E-01&-4.58 \\
\hline
\multicolumn{9}{|c|}{$V^0_5(T)=\text{span}\{ [1;0],\,[0;1],\,\{\mbox{ReLU}~\bm{\sigma}(t_i)\}_{i=1}^{4} \}$~and~$\gamma=-2$}\\
\hline
1/8&2.09E-02&0.81 &3.32E-02&0.76 &3.86E-02&0.08 &4.36E-02&-0.46 \\
1/16&1.23E-02&0.76 &1.95E-02&0.77 &3.51E-02&0.14 &3.64E-02&0.26 \\
1/32&1.43E-02&-0.22 &2.07E-02&-0.09 &3.77E-02&-0.10 &4.32E-02&-0.25 \\
1/64&1.26E-02&0.18 &1.80E-02&0.21 &3.28E-02&0.20 &3.40E-02&0.34 \\
\hline
\end{tabular}
\end{table}

\begin{exmp}\label{locking free EX}
To investigate the locking-free property, we consider the exact solution
$$
\bm{u}=\left[\begin{matrix}\sin(x) \sin(y) + \frac{x}{\lambda} \\ \cos(x) \cos(y) + \frac{y}{\lambda}\end{matrix}\right].
$$
Tables \ref{locking_free_EX_Rectangle_1}--\ref{locking_free_EX_Tri_3} report the corresponding errors and convergence rates including details on the partition type, choice of $V^0(T)$, $V^b(T)$, $\gamma$, $R_b$, as well as the Lam\'e constants $\mu$ and $\lambda$. 
\end{exmp}

\smallskip
Table \ref{locking_free_EX_Rectangle_1} illustrating the locking-free behavior of the gWG method for both the polynomial spaces and activation-based function spaces with arbitrary $t_{i}$. As $\lambda$ increases, the convergence rates of $V^0_3(T)$ - $V^0_6(T)$ for certain error norms exhibits a slight reduction. Interestingly, for $V^0_7(T)$ and $V^0_8(T)$, the convergence rates for $\lambda=10^6$ are actually better than those for $\lambda=1$, demonstrating the robustness of the method against locking effects.

\begin{table}[h]
\caption{Errors and convergence rates for gWG space $V^0(T)|[P_0(\partial T)]^2|[P_0(T)]^{2\times2}|P_0(T)$ on rectangular elements for Example \ref{locking free EX}; $R_b=Q_b$; $\mu=0.5$.}\label{locking_free_EX_Rectangle_1}
\small
\centering
\begin{tabular}{|c c c c c c c c c|}
\hline
$h$&$\|\bm{u}-\bm{u}_0\|$&Rate&$\|\bm{u}-\bm{u}_b\|$&Rate&$\|\bm{u}-\bm{u}_0\|_{\infty}$&Rate&$\|\bm{u}-\bm{u}_b\|_{\infty}$&Rate\\
\hline
\multicolumn{9}{|c|}{$V^0_1(T)=[P_1(T)]^2$,~$\gamma=-1$~and~$\lambda=1$}\\
\hline
1/8&5.79E-03&1.89 &6.44E-02&1.09 &7.38E-03&1.80 &9.07E-03&1.75 \\
1/16&1.49E-03&1.96 &3.12E-02&1.05 &1.96E-03&1.92 &2.41E-03&1.91 \\
1/32&3.76E-04&1.99 &1.53E-02&1.02 &5.04E-04&1.96 &6.46E-04&1.90 \\
1/64&9.41E-05&2.00 &7.60E-03&1.01 &1.28E-04&1.98 &1.70E-04&1.92 \\
\hline
\multicolumn{9}{|c|}{$V^0_2(T)=[P_1(T)]^2$,~$\gamma=-1$~and~$\lambda=10^6$}\\
\hline
1/8&7.30E-03&1.78 &3.52E-02&1.12 &8.72E-03&1.68 &1.03E-02&1.71 \\
1/16&1.96E-03&1.90 &1.68E-02&1.07 &2.36E-03&1.88 &2.78E-03&1.89 \\
1/32&5.02E-04&1.96 &8.19E-03&1.03 &6.05E-04&1.97 &7.10E-04&1.97 \\
1/64&1.27E-04&1.99 &4.05E-03&1.01 &1.52E-04&1.99 &1.78E-04&1.99 \\
\hline
\multicolumn{9}{|c|}{$V^0_3(T)=\text{span}\{ [1;0],\,[0;1],\,\{\mbox{sin}~\bm{\sigma}(t_i)\}_{i=1}^{4} \}$,~$\gamma=-1$~and~$\lambda=1$}\\
\hline
1/8&5.79E-03&1.96 &6.44E-02&1.09 &7.33E-03&1.99 &9.02E-03&1.85 \\
1/16&1.49E-03&1.96 &3.12E-02&1.05 &1.96E-03&1.90 &2.44E-03&1.89 \\
1/32&3.76E-04&1.99 &1.53E-02&1.02 &5.04E-04&1.96 &6.46E-04&1.92 \\
1/64&9.54E-05&1.98 &7.60E-03&1.01 &1.53E-04&1.72 &2.33E-04&1.47 \\
\hline
\multicolumn{9}{|c|}{$V^0_4(T)=\text{span}\{ [1;0],\,[0;1],\,\{\mbox{sin}~\bm{\sigma}(t_i)\}_{i=1}^{4} \}$,~$\gamma=-1$~and~$\lambda=10^6$}\\
\hline
1/8&1.33E-02&1.78 &3.76E-02&1.13 &1.45E-02&1.48 &3.00E-02&1.14 \\
1/16&4.85E-03&1.46 &1.80E-02&1.06 &6.51E-03&1.16 &1.59E-02&0.92 \\
1/32&1.88E-03&1.37 &8.71E-03&1.05 &2.02E-03&1.69 &8.22E-03&0.95 \\
1/64&8.53E-04&1.14 &4.30E-03&1.02 &1.10E-03&0.88 &4.50E-03&0.87 \\
\hline
\multicolumn{9}{|c|}{$V^0_5(T)=\text{span}\{ [1;0],\,[0;1],\,\{\mbox{cos}~\bm{\sigma}(t_i)\}_{i=1}^{4} \}$,~$\gamma=0$~and~$\lambda=1$}\\
\hline
1/8&6.20E-02&0.89 &6.84E-02&1.05 &4.57E-02&0.41 &4.25E-02&0.46 \\
1/16&3.21E-02&0.95 &3.46E-02&0.98 &2.56E-02&0.84 &2.57E-02&0.72 \\
1/32&1.66E-02&0.95 &1.76E-02&0.98 &1.38E-02&0.89 &1.51E-02&0.77 \\
1/64&8.47E-03&0.97 &8.90E-03&0.98 &7.46E-03&0.89 &8.55E-03&0.82 \\
\hline
\multicolumn{9}{|c|}{$V^0_6(T)=\text{span}\{ [1;0],\,[0;1],\,\{\mbox{cos}~\bm{\sigma}(t_i)\}_{i=1}^{4} \}$,~$\gamma=0$~and~$\lambda=10^6$}\\
\hline
1/8&4.58E-02&0.86 &4.27E-02&0.89 &4.71E-02&0.68 &3.90E-02&0.48 \\
1/16&2.50E-02&0.87 &2.37E-02&0.85 &2.63E-02&0.84 &2.48E-02&0.65 \\
1/32&1.34E-02&0.90 &1.31E-02&0.86 &1.49E-02&0.82 &1.55E-02&0.68 \\
1/64&7.00E-03&0.93 &7.00E-03&0.90 &7.87E-03&0.92 &8.18E-03&0.92 \\
\hline
\multicolumn{9}{|c|}{$V^0_7(T)=\text{span}\{ [1;0],\,[0;1],\,\{\mbox{Sigmoid}~\bm{\sigma}(t_i)\}_{i=1}^{4} \}$,~$\gamma=-1$~and~$\lambda=1$}\\
\hline
1/8&5.81E-03&1.88 &6.44E-02&1.09 &7.37E-03&1.77 &9.05E-03&1.75 \\
1/16&1.49E-03&1.96 &3.12E-02&1.05 &2.05E-03&1.84 &2.42E-03&1.90 \\
1/32&3.76E-04&1.99 &1.53E-02&1.02 &5.04E-04&2.03 &6.46E-04&1.90 \\
1/64&3.53E-04&0.09 &7.62E-03&1.01 &4.15E-04&0.28 &3.44E-03&-2.41 \\
\hline
\multicolumn{9}{|c|}{$V^0_8(T)=\text{span}\{ [1;0],\,[0;1],\,\{\mbox{Sigmoid}~\bm{\sigma}(t_i)\}_{i=1}^{4} \}$,~$\gamma=-1$~and~$\lambda=10^6$}\\
\hline
1/8&1.33E-02&1.45 &3.82E-02&1.06 &1.16E-02&1.54 &3.23E-02&1.01 \\
1/16&4.99E-03&1.41 &1.82E-02&1.07 &6.08E-03&0.93 &1.68E-02&0.95 \\
1/32&1.67E-03&1.58 &8.63E-03&1.08 &1.67E-03&1.87 &8.79E-03&0.93 \\
1/64&6.13E-04&1.45 &4.18E-03&1.05 &9.93E-04&0.75 &4.53E-03&0.96 \\
\hline
\end{tabular}
\end{table}

\smallskip
Table \ref{locking_free_EX_Rectangle_2} illustrates the performance of the gWG algorithm using Leaky-ReLU and ReLU activations. The spaces $V^0_1(T)$ and $V^0_2(T)$, which extend the classical $P_1$ non-conforming element, exhibit clear locking-free behavior. For $V^0_3(T)$ and $V^0_4(T)$ with $\gamma=0$, we observe convergence in $\|u-u_0\|$, $\|u-u_b\|$, and $\|u-u_0\|_{\infty}$ at nearly first-order rates. Regarding the ReLU activation-based function spaces $V^0_5(T)$ - $V^0_8(T)$,  the convergence profiles for $\lambda=10^6$ appear slightly better than those for $\lambda=1$. Compared with $V^0_3(T)$--$V^0_4(T)$ in Table \ref{general_EX_Rectangle_2}, the spaces $V^0_6(T)$ and $V^0_8(T)$ in Table \ref{locking_free_EX_Rectangle_2} achieve the desired convergence rates.

\begin{table}[h]
\caption{Errors and convergence rates for gWG space $V^0(T)|[P_0(\partial T)]^2|[P_0(T)]^{2\times2}|P_0(T)$ on rectangular elements for Example \ref{locking free EX}; $R_b=Q_b$; $\mu=0.5$. }\label{locking_free_EX_Rectangle_2}
\small
\centering
\begin{tabular}{|c c c c c c c c c|}
\hline
$h$&$\|\bm{u}-\bm{u}_0\|$&Rate&$\|\bm{u}-\bm{u}_b\|$&Rate&$\|\bm{u}-\bm{u}_0\|_{\infty}$&Rate&$\|\bm{u}-\bm{u}_b\|_{\infty}$&Rate\\
\hline
\multicolumn{9}{|c|}{$V^0_1(T)=\text{span}\{ [1;0],\,[0;1],\,\{\mbox{L-ReLU}~\bm{\sigma}(t_i)\}_{i=1}^{4} \}$,~$\gamma=-1$,~$\varepsilon=1$~and~$\lambda=1$}\\
\hline
1/8&5.79E-03&1.89 &6.44E-02&1.09 &7.38E-03&1.80 &9.07E-03&1.75 \\
1/16&1.49E-03&1.96 &3.12E-02&1.05 &1.96E-03&1.92 &2.41E-03&1.91 \\
1/32&3.76E-04&1.99 &1.53E-02&1.02 &5.04E-04&1.96 &6.46E-04&1.90 \\
1/64&9.41E-05&2.00 &7.60E-03&1.01 &1.28E-04&1.98 &1.70E-04&1.92 \\
\hline
\multicolumn{9}{|c|}{$V^0_2(T)=\text{span}\{ [1;0],\,[0;1],\,\{\mbox{L-ReLU}~\bm{\sigma}(t_i)\}_{i=1}^{4} \}$,~$\gamma=-1$,~$\varepsilon=1$~and~$\lambda=10^6$}\\
\hline
1/8&7.30E-03&1.78 &3.52E-02&1.12 &8.72E-03&1.68 &1.03E-02&1.71 \\
1/16&1.96E-03&1.90 &1.68E-02&1.07 &2.36E-03&1.88 &2.78E-03&1.89 \\
1/32&5.02E-04&1.96 &8.19E-03&1.03 &6.05E-04&1.97 &7.10E-04&1.97 \\
1/64&1.27E-04&1.99 &4.05E-03&1.01 &1.52E-04&1.99 &1.78E-04&1.99 \\
\hline
\multicolumn{9}{|c|}{$V^0_3(T)=\text{span}\{ [1;0],\,[0;1],\,\{\mbox{L-ReLU}~\bm{\sigma}(t_i)\}_{i=1}^{4} \}$,~$\gamma=0$,~$\varepsilon=1/10$~and~$\lambda=1$}\\
\hline
1/8&5.94E-02&0.93 &6.91E-02&1.02 &4.64E-02&0.65 &4.36E-02&0.18 \\
1/16&3.19E-02&0.90 &3.45E-02&1.00 &2.58E-02&0.85 &2.54E-02&0.78 \\
1/32&1.65E-02&0.95 &1.75E-02&0.98 &1.43E-02&0.85 &1.54E-02&0.72 \\
1/64&8.47E-03&0.97 &8.89E-03&0.98 &7.33E-03&0.96 &7.91E-03&0.96 \\
\hline
\multicolumn{9}{|c|}{$V^0_4(T)=\text{span}\{ [1;0],\,[0;1],\,\{\mbox{L-ReLU}~\bm{\sigma}(t_i)\}_{i=1}^{4} \}$,~$\gamma=0$,~$\varepsilon=1/10$~and~$\lambda=10^6$}\\
\hline
1/8&4.69E-02&0.86 &4.21E-02&0.89 &4.40E-02&0.63 &3.54E-02&0.26 \\
1/16&2.54E-02&0.88 &2.36E-02&0.84 &2.59E-02&0.77 &2.36E-02&0.58 \\
1/32&1.35E-02&0.91 &1.30E-02&0.86 &1.44E-02&0.84 &1.39E-02&0.76 \\
1/64&7.04E-03&0.94 &6.99E-03&0.90 &7.95E-03&0.86 &8.72E-03&0.67 \\
\hline
\multicolumn{9}{|c|}{$V^0_5(T)=\text{span}\{ [1;0],\,[0;1],\,\{\mbox{ReLU}~\bm{\sigma}(t_i)\}_{i=1}^{4} \}$,~$\gamma=-1$~and~$\lambda=1$}\\
\hline
1/8&4.84E-02&1.00 &6.79E-02&1.03 &2.23E-02&1.18 &3.34E-02&0.49 \\
1/16&2.57E-02&0.91 &3.62E-02&0.91 &2.44E-02&-0.13 &2.97E-02&0.17 \\
1/32&1.60E-02&0.69 &2.23E-02&0.70 &2.13E-02&0.20 &2.37E-02&0.32 \\
1/64&1.20E-02&0.42 &1.68E-02&0.41 &2.00E-02&0.09 &2.26E-02&0.07 \\
\hline
\multicolumn{9}{|c|}{$V^0_6(T)=\text{span}\{ [1;0],\,[0;1],\,\{\mbox{ReLU}~\bm{\sigma}(t_i)\}_{i=1}^{4} \}$,~$\gamma=-1$~and~$\lambda=10^6$}\\
\hline
1/8&3.23E-02&1.03 &3.45E-02&1.08 &7.60E-03&0.87 &1.26E-02&-0.54 \\
1/16&1.62E-02&1.00 &1.69E-02&1.04 &4.75E-03&0.68 &9.55E-03&0.40 \\
1/32&8.22E-03&0.97 &8.56E-03&0.98 &3.18E-03&0.58 &4.11E-03&1.22 \\
1/64&4.38E-03&0.91 &4.76E-03&0.85 &5.64E-03&-0.83 &9.40E-03&-1.20 \\
\hline
\multicolumn{9}{|c|}{$V^0_7(T)=\text{span}\{ [1;0],\,[0;1],\,\{\mbox{ReLU}~\bm{\sigma}(t_i)\}_{i=1}^{4} \}$,~$\gamma=0$~and~$\lambda=1$}\\
\hline
1/8&6.25E-02&0.90 &6.87E-02&1.02 &4.27E-02&0.67 &3.64E-02&0.38 \\
1/16&3.26E-02&0.94 &3.48E-02&0.98 &2.57E-02&0.73 &4.17E-02&-0.19 \\
1/32&1.67E-02&0.97 &1.75E-02&0.99 &1.38E-02&0.89 &1.37E-02&1.61 \\
1/64&1.27E-02&0.39 &1.76E-02&0.00 &2.22E-01&-4.01 &3.47E-01&-4.67 \\
\hline
\multicolumn{9}{|c|}{$V^0_8(T)=\text{span}\{ [1;0],\,[0;1],\,\{\mbox{ReLU}~\bm{\sigma}(t_i)\}_{i=1}^{4} \}$,~$\gamma=0$~and~$\lambda=10^6$}\\
\hline
1/8&4.69E-02&0.86 &4.21E-02&0.89 &4.42E-02&0.62 &3.60E-02&0.24 \\
1/16&2.54E-02&0.88 &2.36E-02&0.84 &2.59E-02&0.77 &2.36E-02&0.61 \\
1/32&1.35E-02&0.91 &1.30E-02&0.86 &1.44E-02&0.84 &1.39E-02&0.76 \\
1/64&7.04E-03&0.94 &6.99E-03&0.90 &7.82E-03&0.88 &8.08E-03&0.79 \\
\hline
\end{tabular}
\end{table}

\smallskip
Table \ref{locking_free_EX_Tri_3} shows the performance of the gWG algorithm for $R_b=I$ and $R_b=Q_b$ on triangular partitions. Comparing the cases $V^0_i(T)$ $(i=3,\,5,\,7)$ in Table \ref{general_EX_Triangle} with $V^0_j(T)$ $(j=1,2)$ in Table \ref{locking_free_EX_Tri_3}, we observe that the numerical results deteriorate when $R_b=Q_b$ is used on triangular elements. In contrast, the spaces $V^0_3(T)$--$V^0_8(T)$ in Table \ref{locking_free_EX_Tri_3} achieve the expected convergence rates when $R_b=I$, demonstrating the stability of the identity operator and confirming the locking-free property of the gWG method for linear elasticity problems.

\begin{table}[h]
\caption{Errors and convergence rates for gWG space $V^0(T)|[P_0(\partial T)]^2|[P_0(T)]^{2\times2}|P_0(T)$ on triangular elements for Example \ref{locking free EX}; $\gamma=0$; $\mu=0.5$. }\label{locking_free_EX_Tri_3}
\small
\centering
\begin{tabular}{|c c c c c c c c c|}
\hline
$h$&$\|\bm{u}-\bm{u}_0\|$&Rate&$\|\bm{u}-\bm{u}_b\|$&Rate&$\|\bm{u}-\bm{u}_0\|_{\infty}$&Rate&$\|\bm{u}-\bm{u}_b\|_{\infty}$&Rate\\
\hline
\multicolumn{9}{|c|}{$V^0_1(T)=[P_1(T)]^2$,~$R_b=Q_b$~and~$\lambda=1$}\\
\hline
1/8&2.55E-02&0.48 &1.07E-01&0.97 &2.74E-02&0.47 &4.34E-02&0.06 \\
1/16&2.17E-02&0.24 &6.11E-02&0.80 &2.09E-02&0.39 &4.35E-02&0.00 \\
1/32&2.03E-02&0.09 &4.34E-02&0.49 &1.75E-02&0.25 &4.35E-02&0.00 \\
1/64&1.99E-02&0.03 &3.78E-02&0.20 &1.58E-02&0.15 &4.35E-02&0.00 \\
\hline
\multicolumn{9}{|c|}{$V^0_2(T)=[P_1(T)]^2$,~$R_b=Q_b$~and~$\lambda=10^6$}\\
\hline
1/8&2.62E-02&0.46 &5.64E-02&0.78 &3.08E-02&0.41 &4.63E-02&-0.02 \\
1/16&2.22E-02&0.24 &4.16E-02&0.44 &2.48E-02&0.31 &4.72E-02&-0.03 \\
1/32&2.07E-02&0.10 &3.71E-02&0.16 &2.15E-02&0.20 &4.73E-02&0.00 \\
1/64&2.01E-02&0.04 &3.60E-02&0.05 &1.99E-02&0.12 &4.73E-02&0.00 \\
\hline
\multicolumn{9}{|c|}{$V^0_3(T)=[P_1(T)]^2$,~$R_b=I$~and~$\lambda=1$}\\
\hline
1/8&2.82E-02&0.97 &1.01E-01&1.03 &1.96E-02&0.75 &1.26E-02&0.57 \\
1/16&1.43E-02&0.98 &5.03E-02&1.01 &1.11E-02&0.83 &8.39E-03&0.59 \\
1/32&7.24E-03&0.98 &2.51E-02&1.00 &6.03E-03&0.88 &5.18E-03&0.70 \\
1/64&3.65E-03&0.99 &1.26E-02&1.00 &3.19E-03&0.92 &2.93E-03&0.82 \\
\hline
\multicolumn{9}{|c|}{$V^0_4(T)=[P_1(T)]^2$,~$R_b=I$~and~$\lambda=10^6$}\\
\hline
1/8&1.89E-02&0.88 &4.57E-02&1.01 &1.96E-02&0.74 &1.32E-02&0.49 \\
1/16&1.00E-02&0.91 &2.31E-02&0.98 &1.11E-02&0.82 &8.38E-03&0.66 \\
1/32&5.27E-03&0.93 &1.18E-02&0.97 &6.09E-03&0.87 &4.90E-03&0.77 \\
1/64&2.73E-03&0.95 &6.03E-03&0.97 &3.23E-03&0.91 &2.79E-03&0.81 \\
\hline
\multicolumn{9}{|c|}{$V^0_5(T)=\text{span}\{ [1;0],\,[0;1],\,\{\mbox{sin}~\bm{\sigma}(t_i)\}_{i=1}^{4} \}$,~$R_b=I$~and~$\lambda=1$}\\
\hline
1/8&2.82E-02&0.97 &1.01E-01&1.03 &1.96E-02&0.75 &1.26E-02&0.57 \\
1/16&1.43E-02&0.98 &5.03E-02&1.01 &1.11E-02&0.83 &8.40E-03&0.59 \\
1/32&7.24E-03&0.98 &2.51E-02&1.00 &6.03E-03&0.88 &5.18E-03&0.70 \\
1/64&3.65E-03&0.99 &1.26E-02&1.00 &3.19E-03&0.92 &2.93E-03&0.82 \\
\hline
\multicolumn{9}{|c|}{$V^0_6(T)=\text{span}\{ [1;0],\,[0;1],\,\{\mbox{sin}~\bm{\sigma}(t_i)\}_{i=1}^{4} \}$,~$R_b=I$~and~$\lambda=10^6$}\\
\hline
1/8&2.03E-02&1.01 &4.53E-02&1.02 &1.97E-02&0.74 &1.33E-02&0.45 \\
1/16&1.03E-02&0.98 &2.30E-02&0.98 &1.11E-02&0.83 &8.14E-03&0.71 \\
1/32&5.30E-03&0.95 &1.17E-02&0.97 &6.09E-03&0.86 &4.91E-03&0.73 \\
1/64&2.73E-03&0.95 &6.01E-03&0.97 &3.24E-03&0.91 &2.75E-03&0.83 \\
\hline
\multicolumn{9}{|c|}{$V^0_7(T)=\text{span}\{ [1;0],\,[0;1],\,\{\mbox{Sigmoid}~\bm{\sigma}(t_i)\}_{i=1}^{4} \}$,~$R_b=I$~and~$\lambda=1$}\\
\hline
1/8&2.82E-02&0.97 &1.01E-01&1.03 &1.96E-02&0.75 &1.26E-02&0.57 \\
1/16&1.43E-02&0.98 &5.03E-02&1.01 &1.11E-02&0.83 &8.39E-03&0.59 \\
1/32&7.24E-03&0.98 &2.51E-02&1.00 &6.03E-03&0.88 &5.18E-03&0.70 \\
1/64&3.65E-03&0.99 &1.26E-02&1.00 &3.19E-03&0.92 &2.93E-03&0.82 \\
\hline
\multicolumn{9}{|c|}{$V^0_8(T)=\text{span}\{ [1;0],\,[0;1],\,\{\mbox{Sigmoid}~\bm{\sigma}(t_i)\}_{i=1}^{4} \}$,~$R_b=I$~and~$\lambda=10^6$}\\
\hline
1/8&1.96E-02&0.97 &4.54E-02&1.02 &1.95E-02&0.75 &1.27E-02&0.56 \\
1/16&1.01E-02&0.95 &2.30E-02&0.98 &1.11E-02&0.81 &8.37E-03&0.60 \\
1/32&5.28E-03&0.94 &1.18E-02&0.97 &6.08E-03&0.87 &4.91E-03&0.77 \\
1/64&2.73E-03&0.95 &6.02E-03&0.97 &3.26E-03&0.90 &2.78E-03&0.82 \\
\hline
\end{tabular}
\end{table}

\section{Conclusions}
This paper introduced a generalized weak Galerkin (gWG) finite element method for linear elasticity problems on general polygonal and polyhedral meshes. The proposed new algorithm is based on a generalized weak differential operator that combines the classical differential operator with a local correction, distinguishing it from traditional weak Galerkin formulations. This construction enables the use of arbitrary finite-dimensional approximation spaces. We derived the associated error equations and established rigorous a priori error estimates.

Numerical experiments demonstrate the effectiveness of the method for both activation-function-based approximation spaces with randomly generated parameters and classical polynomial spaces on triangular and rectangular meshes. The results confirm that the proposed method is locking-free. In several test cases, activation-based spaces achieve convergence rates comparable to those of polynomial spaces, consistent with the theoretical predictions. These results highlight the flexibility, robustness, and broad applicability of the gWG framework for linear elasticity problems.

\bigskip
\noindent
\textbf{Acknowledgements}. The research of Junping Wang was supported by the NSF IR/D program while the author is working at the U.S. National Science Foundation. Any opinions, findings, and conclusions or recommendations expressed in this material are those of the author and do not necessarily reflect the views of the U.S. National Science Foundation.

\bibliographystyle{unsrt}
\bibliography{ref}

\end{document}